\documentclass[final,3p,times]{elsarticle}

\usepackage{graphicx}%
\usepackage{multirow}%
\usepackage{amsmath,amssymb,amsfonts}%
\usepackage{amsthm}%
\usepackage{mathrsfs}%
\usepackage[title]{appendix}%
\usepackage{xcolor}%
\usepackage{textcomp}%
\usepackage{manyfoot}%
\usepackage{booktabs}%
\usepackage{algorithm}%
\usepackage{algorithmicx}%
\usepackage{algpseudocode}%
\usepackage{listings}%
\usepackage{xcolor}
\newcommand{\corr}[1]{\textcolor{black}{#1}}

\usepackage{tikz}
\usepackage{pgfplots}
\pgfplotsset{compat=newest}

\theoremstyle{definition}
\newtheorem{definition}{Definition}

\theoremstyle{plain}
\newtheorem{theorem}{Theorem}[section]
\newtheorem{lemma}{Lemma}[section]
\newtheorem{remark}{Remark}[section]

\newtheorem{example}{Example}

\numberwithin{equation}{section}

\usepackage[english]{babel}

\journal{Applied Numerical Mathematics}

\begin{document}

\begin{frontmatter}

\title{A nonstandard fractional Jacobi pseudospectral method for weakly regular right-sided Caputo terminal problems}

\author[f1]{Ibrahem G. Ameen}
  \ead{ibrahem\_gamal.201@azhar.edu.eg}

\author[f2]{Magda Hammad}
  \ead{magda.hammad@science.bsu.edu.eg}
\author[f3]{Mohsen Razzaghi}
  \ead{razzaghi@math.msstate.edu}
\author[f4]{Mahmoud A. Zaky$^{*,}$}
  \ead{ma.zaky@yahoo.com;mibrahimm@imamu.edu.sa}

\address[f1]{Department of Mathematics, Faculty of Science, Al-Azhar University, Cairo, Egypt}

\address[f2]{Mathematics and Computer Science Department, Faculty of Science, Beni-Suef University, Beni-Suef 62511, Egypt}

\address[f3]{Department of Mathematics and Statistics, Mississippi State University,   Mississippi State, MS 39762, USA}

\address[f4]{Department of Mathematics and Statistics, College of Science, Imam Mohammad Ibn Saud Islamic University (IMSIU), Riyadh, Saudi Arabia}

\cortext[mycorrespondingauthor]{Corresponding author: M.A. Zaky}

\begin{abstract}
Right-sided fractional differential equations arising in terminal value problems present substantial analytical and numerical challenges due to their intrinsic nonlocal structure and the presence of weak algebraic-type singularities near the terminal point. These singular behaviors significantly degrade the performance of classical finite-difference, finite-element, and standard spectral methods, which typically rely on smoothness assumptions that break down in the fractional setting. In this work, we develop a comprehensive fractional pseudospectral framework specifically designed for nonlinear right-sided Caputo terminal value systems. The proposed method employs a versatile family of nonstandard fractional Jacobi basis functions whose endpoint-weighted structure naturally reflects the asymptotic behavior of right-sided fractional operators. This construction enables the accurate approximation of solutions with limited regularity near the terminal point, without requiring smoothing transformations or mesh grading. We establish the existence and uniqueness of the discrete numerical solution and derive rigorous convergence estimates that faithfully capture the inherent singularity structure of the problem. A series of numerical experiments validates the theoretical developments and demonstrates that the proposed approach delivers high-order accuracy and superior resolution of terminal-layer behavior compared with traditional spectral techniques. The resulting framework provides an efficient and reliable computational tool for challenging nonlinear fractional terminal value systems characterized by weak solution regularity.

\end{abstract}


\end{frontmatter}

%\linenumbers
\section{Introduction}

Fractional-order systems are dynamical systems governed by differential equations involving derivatives of non-integer order. Such systems naturally incorporate memory and hereditary effects, which makes them effective tools for describing real-world processes whose future states depend not only on the current configuration but also on the full evolution of the system \cite{podlubny1998fractional}. In contrast to classical integer-order models, fractional formulations offer greater flexibility in capturing anomalous diffusion, long-range temporal dependence, and power-law behavior. A distinctive feature of many fractional-order models is that their solutions often possess low regularity or nonsmooth behavior, even when the coefficients and source terms are smooth. This phenomenon is closely related to the nonlocal character of fractional derivatives, since the solution at a given point depends on the entire history of the process.

For right-sided fractional differential equations, the singular behavior typically appears near the terminal point of the interval. Indeed, for right-sided operators, the kernel $(s-y)^{-\beta}$ becomes weakly singular as $y \to 1^{-}$. Consequently, the solution frequently exhibits endpoint singularities, often of algebraic type, near the terminal boundary. Such singular behavior prevents the solution from being classically differentiable up to the endpoint and may lead to sharp gradients or reduced smoothness. As a result, standard numerical methods often suffer from loss of accuracy and nonuniform convergence when applied to these problems \cite{tarasov2011fractional,zayernouri2013fractional}.

These weak singularities have a substantial effect on numerical approximation. Classical spectral and finite-difference methods may experience serious order deterioration, or even fail to converge efficiently, unless suitable modifications are introduced. To resolve the asymptotic behavior near the right endpoint, one must employ specialized basis functions, graded meshes, or fractional polynomial spaces that intrinsically reflect the endpoint singular structure. Such strategies are essential for achieving high-order accuracy, stability, and robustness when solving right-sided fractional models with limited regularity near the terminal boundary \cite{zayernouri1945spectral}.

Recent developments in numerical analysis have produced several accurate and efficient schemes for integral and differential equations involving nonlocal operators and low-regularity solutions. Mokhtary \cite{mokhtary2020computational} developed a Jacobi spectral Galerkin method for Volterra--Hammerstein integral equations with weakly singular kernels and proportional delay, obtaining high spectral accuracy and favorable convergence properties. For boundary value problems with singular solutions, Sgura and Simoncini \cite{sgura2014finite} introduced high-order exponentially convergent discrete fractional techniques based on a hybrid strategy capable of treating both known and unknown singularities. Caboussat and Rappaz \cite{caboussat2022least} proposed a least-squares/relaxation finite element method for the prescribed Jacobian equation, combining local nonlinear iterations with variational linear solvers to achieve accurate and stable numerical approximations. Zaky \cite{zaky2020accurate} established an adaptive spectral collocation framework for nonlinear rational-order fractional initial value problems by employing a smoothing transformation that resolves initial singularities and recovers exponential convergence.

Further progress has been reported for systems and variable-coefficient fractional problems. Faghih and Mokhtary \cite{faghih2022non} investigated nonlinear systems of multi-order fractional differential equations, established well-posedness, and designed a high-order operational spectral Galerkin method with spectral accuracy and reduced computational complexity. Zheng et al. \cite{zheng2020numerical} proposed and analyzed a spectral method for variable-coefficient fractional diffusion equations by transforming them into equivalent constant-coefficient problems and deriving rigorous error estimates. In a related direction, stable and efficient spectral Galerkin schemes based on Müntz--Jacobi functions were introduced in \cite{shiri2021terminal,faghih2025spectral}. Zaky et al. \cite{zaky2025improved,zaky2020semi} developed nonstandard spectral collocation methods based on the second-kind fractional Jacobi functions (FJFs2) for nonlinear fractional differential equations, achieving high-order accuracy for low-regularity solutions through weighted $\mathcal{L}_2$ analysis and extensive numerical validation. Moreover, Zaky and Ameen \cite{zaky2020priori,ameen2021pseudo,zaky2024non} introduced a Legendre--Jacobi collocation approach for nonlinear systems of two-point boundary value problems formulated through equivalent Volterra--Fredholm integral equations, and established spectral convergence for both smooth and nonsmooth solutions. Tian et al. \cite{tian2022high} developed a high-order finite-difference method for right-sided Caputo (R-SC)  fractional differential equations. Zhou and Xu \cite{zhou2016well} analyzed a nonlinear coupled system of fractional differential equations and proved existence and uniqueness by transforming the problem into an equivalent system of integral equations and applying fixed-point arguments. \textcolor{black}{It is worth noting that fast time-stepping methods have been developed for
left-sided (initial-value type) Caputo derivatives in order to alleviate the
memory and computational burden associated with their nonlocal, history-dependent
structure; see, e.g., \cite{jiang2017fast,yan2017fast,zhang2022fast,fang2020fast,huang2022unified}.
The present method, being a global pseudospectral scheme rather than a
step-by-step marching procedure, is formulated in a fundamentally different manner
and does not rely on a time-history recursion of this type. Nevertheless, whether
fast evaluation strategies analogous to those developed for left-sided operators
can be formulated for right-sided fractional operators of the type considered in
this work remains a question of independent interest, which we comment on further
in the concluding section.}

Despite these advances, fractional terminal value problems remain particularly challenging from the numerical point of view. The presence of R-SC derivatives generates nonlocal memory effects and typically induces intrinsic singularities near the terminal point, even when the forcing terms are smooth. These singularities significantly reduce the effectiveness of standard discretization techniques and often cause severe order reduction. Although spectral methods are well known for their high accuracy, their performance deteriorates in the presence of endpoint singularities because global polynomial bases are not well suited to approximating functions with limited regularity. The difficulty becomes even greater for coupled nonlinear systems, where the singular behavior propagates through all components of the solution. Therefore, it is necessary to construct approximation spaces and discretization strategies that are specifically tailored to terminal singularities while maintaining high-order accuracy for nonsmooth fractional dynamics.

In this work, we study the following nonlinear fractional terminal value system \cite{jankowski2015systems}:
\textcolor{black}{\begin{equation}\label{eq1.1} \begin{cases} {}_{t}^{C}D_{1}^{\beta}\, U_1(t) = F_1\!\big(t, U_1(t), U_2(t), \ldots, U_Q(t)\big), \quad U_1(1)=0,\\[4pt] {}_{t}^{C}D_{1}^{\beta}\, U_2(t) = F_2\!\big(t, U_1(t), U_2(t), \ldots, U_Q(t)\big), \quad U_2(1)=0,\\[4pt] \hspace{1.2cm}\vdots \\[4pt] {}_{t}^{C}D_{1}^{\beta}\, U_Q(t) = F_Q\!\big(t, U_1(t), U_2(t), \ldots, U_Q(t)\big), \quad U_Q(1)=0, \end{cases} \end{equation}}
%\begin{equation}\label{eq1.1} \begin{cases} {}_{y}^{C}D_{1}^{\beta} U_q(y) = F_q\!\big(y,U_1(y),U_2(y),\ldots,U_Q(y)\big), \quad y\in [0,1), \quad q=1,2,\ldots,Q,\\[4pt] U_q(1)=0, \quad q=1,2,\ldots,Q, \end{cases} \end{equation}
where $\beta\in(0,1)$. Here, the R-SC fractional derivative is defined by
\begin{equation}\label{eq1.2a}
{}_{t}^{C}D_{1}^{\beta}U(t)
= -\frac{1}{\Gamma(1-\beta)}
\int_{t}^{1}(s-t)^{-\beta}U'(s)\,ds.
\end{equation}

\textcolor{black}{\begin{remark}
The homogeneous terminal condition in \eqref{eq1.1} is assumed without loss of
generality. For a non-homogeneous condition $U_q(1)=U_q^{T}$, the shift
$\widetilde U_q:=U_q-U_q^{T}$ reduces \eqref{eq1.1} to a system of identical form
with $\widetilde U_q(1)=0$, since ${}_{t}^{C}D_{1}^{\beta}U_q^{T}=0$, and the
shifted right-hand side inherits the smoothness and Lipschitz properties of $F_q$.
Hence Theorem~\ref{the22} and Theorem~\ref{thm:exist-unique} apply unchanged, with
$U_q(t)=\widetilde U_q(t)+U_q^{T}$. We adopt the homogeneous case throughout for
notational simplicity.
\end{remark}}

Building on the high-order fractional collocation framework developed in \cite{zaky2025high}, the present work extends that methodology to fully coupled nonlinear terminal systems and introduces a more flexible nonstandard Jacobi basis capable of resolving the severe terminal singularities arising in right-sided fractional dynamics. The main objective is to construct a robust and highly accurate numerical method for solving the nonlinear system \eqref{eq1.1}. In particular, the proposed approach is designed to accommodate terminal singularities naturally, retain high-order accuracy for nonsmooth solutions, and provide a rigorous convergence analysis for the coupled nonlinear setting. To achieve this, we develop a fractional spectral collocation method based on nonstandard fractional Jacobi basis functions. Owing to their endpoint-weighted structure, these basis functions are well matched to the asymptotic behavior of right-sided fractional operators and therefore provide an effective approximation space for solutions of limited regularity. As a consequence, the proposed method overcomes the principal limitations of standard polynomial-based spectral schemes and yields improved accuracy and resolution near the terminal layer.

The remainder of this paper is organized as follows. Section~\ref{sec2} introduces the analytical framework and the essential properties of the FJFs2 that form the basis of the approximation space. Section~\ref{sec3} presents the construction of the proposed spectral collocation scheme and its discrete formulation. Section~\ref{sec:conv} is devoted to the convergence analysis. In Section~\ref{secx5}, we prove that the resulting discrete problem admits a unique numerical solution. Finally, Section~\ref{sec6} reports numerical experiments that confirm the theoretical analysis and demonstrate the accuracy, stability, and computational effectiveness of the proposed method.

\section{Properties of FJFs2}\label{sec2}

In this section, we introduce the family of FJFs2 and summarize their main analytical properties. Throughout, we work on the interval $
\chi:=(0,1),$ 
and fix a fractional scaling parameter $\mu\in(0,1]$. We begin with the FJFs2 space generated by fractional powers of $(1-\varsigma)$:
\begin{equation}\label{spacLR}
\mathbb{S}_{i}^{\mu}
:= \mathrm{span}\big\{1,(1-\varsigma)^{\mu},\ldots,(1-\varsigma)^{i\mu}\big\},
\qquad
\mathbb{S}^{\mu}:=\bigcup_{i=0}^{\infty}\mathbb{S}_{i}^{\mu}.
\end{equation}

\begin{definition}
Let $\zeta,\kappa>-1$ and $\mu\in(0,1]$. The $i$th-order FJFs2 on $\chi$ is defined by
\begin{equation}
\mathfrak{P}_{i}^{\zeta,\kappa,\mu}(\varsigma)
:= \mathscr{P}_{i}^{\zeta,\kappa}\!\big(1-2(1-\varsigma)^{\mu}\big)
= \sum_{k=0}^{i}
\mathfrak{a}_{k,i}^{\zeta,\kappa}\,(1-\varsigma)^{\mu k},
\qquad i=0,1,2,\ldots,\quad \varsigma\in\chi.
\end{equation}
In the special case $\mu=1$, the above family reduces to the classical shifted Jacobi polynomials:
\[
\mathfrak{P}_{i}^{\zeta,\kappa,1}(\varsigma)
= \mathscr{P}_{i}^{\zeta,\kappa}(2\varsigma-1).
\]
The coefficients $\mathfrak{a}_{k,i}^{\zeta,\kappa}$ are given explicitly by
\begin{equation}
\mathfrak{a}_{k,i}^{\zeta,\kappa}
=
\frac{(-1)^k\,\Gamma(i+\zeta+1)\,
\Gamma(i+k+\zeta+\kappa+1)}
{k!\,(i-k)!\,\Gamma(i+\zeta+\kappa+1)\,
\Gamma(k+\zeta+1)}.
\end{equation}
\end{definition}

\begin{definition}
Let
\[
\eta = 1-(1-\varsigma)^{\mu},
\qquad \varsigma\in\chi.
\]
For a sufficiently smooth function $\Phi$ on $\chi$, we define the first transformed derivative by
\begin{equation*}
\mathfrak{D}_{\mu}^{(1)}\Phi(\varsigma)
:=
\frac{d}{d\eta}\Big(\Phi\big(1-(1-\eta)^{1/\mu}\big)\Big)\Big|_{\eta=1-(1-\varsigma)^{\mu}}
=
\frac{1}{\mu(1-\varsigma)^{\mu-1}}\,
\frac{d\Phi(\varsigma)}{d\varsigma}.
\end{equation*}
For any integer $s\ge1$, the higher-order transformed derivatives are defined recursively by
\begin{equation*}
\mathfrak{D}_{\mu}^{(s)}\Phi(\varsigma)
:=
\underbrace{\mathfrak{D}_{\mu}^{(1)}
\circ \mathfrak{D}_{\mu}^{(1)}
\circ \cdots \circ
\mathfrak{D}_{\mu}^{(1)}}_{s\text{ times}}\Phi(\varsigma).
\end{equation*}
Thus, $\mathfrak{D}_{\mu}^{(s)}$ represents the $s$th-order derivative induced by the fractional change of variables.
\end{definition}
\begin{lemma}
The FJFs2
$\{\mathfrak{P}_{i}^{\zeta,\kappa,\mu}(\varsigma)\}_{i\ge0}$
constitute an orthogonal system on $\chi$ with respect to the weight function
\[
\varpi^{\zeta,\kappa,\mu}(\varsigma)
=
\mu\bigl(1-(1-\varsigma)^{\mu}\bigr)^{\kappa}
(1-\varsigma)^{(\zeta+1)\mu-1},
\qquad
\zeta,\kappa>-1,\quad 0<\mu\le 1.
\]
That is, for all $i,m\ge0$,
\begin{equation}\label{dddddddl.1}
\int_{\chi}
\mathfrak{P}_{i}^{\zeta,\kappa,\mu}(\varsigma)\,
\mathfrak{P}_{m}^{\zeta,\kappa,\mu}(\varsigma)\,
\varpi^{\zeta,\kappa,\mu}(\varsigma)\,d\varsigma
=
\mathfrak{N}_{i}^{\zeta,\kappa}\,\delta_{im},
\end{equation}
where $\delta_{im}$ is the Kronecker delta, and
\[
\mathfrak{N}_{i}^{\zeta,\kappa}
:=
\left\|\mathfrak{P}_{i}^{\zeta,\kappa,\mu}\right\|_{\varpi^{\zeta,\kappa,\mu}}^{2}
=
\frac{\Gamma(i+\kappa+1)\Gamma(i+\zeta+1)}
{i!\,\Gamma(i+\zeta+\kappa+1)\,(2i+\zeta+\kappa+1)}.
\]
Here,
$\|\cdot\|_{\varpi^{\zeta,\kappa,\mu}}$
denotes the weighted $\mathcal{L}^{2}(\chi)$-norm induced by the weight
$\varpi^{\zeta,\kappa,\mu}$.
\end{lemma}

Let $\{\varsigma_{i}^{\zeta,\kappa,\mu},\,w_{i}^{\zeta,\kappa,\mu}\}_{i=0}^{M}$ be the FJFs2--Gauss quadrature nodes and corresponding weights associated with the weight function $\varpi^{\zeta,\kappa,\mu}$. Then, for every $\Psi\in\mathbb{S}_{2M+1}^{\mu}$, the following quadrature formula holds:
\begin{equation}\label{3.3a}
\int_{\chi}\Psi(\varsigma)\,
\varpi^{\zeta,\kappa,\mu}(\varsigma)\,d\varsigma
=
\sum_{i=0}^{M}
\Psi\big(\varsigma_{i}^{\zeta,\kappa,\mu}\big)\,
w_{i}^{\zeta,\kappa,\mu}.
\end{equation}
In particular, the discrete orthogonality relation reads
\begin{equation}\label{3.4a}
\sum_{k=0}^{M}
\mathfrak{P}_{i}^{\zeta,\kappa,\mu}\big(\varsigma_{k}^{\zeta,\kappa,\mu}\big)\,
\mathfrak{P}_{j}^{\zeta,\kappa,\mu}\big(\varsigma_{k}^{\zeta,\kappa,\mu}\big)\,
w_{k}^{\zeta,\kappa,\mu}
=
\mathfrak{N}_{i}^{\zeta,\kappa}\,\delta_{ij},
\qquad 0\le i+j\le 2M+1.
\end{equation}

We next define the associated interpolation operator.

\begin{definition}
The FJFs2--Gauss interpolation operator $\mathcal{I}_{\varsigma,M}^{\zeta,\kappa,\mu}$ is defined by the nodal conditions
\begin{equation}\label{35a}
\mathcal{I}_{\varsigma,M}^{\zeta,\kappa,\mu}\Phi
\big(\varsigma_{s}^{\zeta,\kappa,\mu}\big)
=
\Phi\big(\varsigma_{s}^{\zeta,\kappa,\mu}\big),
\qquad 0\le s\le M.
\end{equation}
For any $\Phi\in\mathbb{S}_{M}^{\mu}$, one has
\[
\mathcal{I}_{\varsigma,M}^{\zeta,\kappa,\mu}\Phi=\Phi.
\]
Since $\mathcal{I}_{\varsigma,M}^{\zeta,\kappa,\mu}\Phi\in\mathbb{S}_{M}^{\mu}$, it admits the expansion
\begin{equation}\label{36a}
\mathcal{I}_{\varsigma,M}^{\zeta,\kappa,\mu}\Phi(\varsigma)
=
\sum_{i=0}^{M}
\widehat{\Phi}_{i}^{\zeta,\kappa,\mu}\,
\mathfrak{P}_{i}^{\zeta,\kappa,\mu}(\varsigma),
\end{equation}
where
\begin{equation}
\widehat{\Phi}_{i}^{\zeta,\kappa,\mu}
=
\frac{1}{\mathfrak{N}_{i}^{\zeta,\kappa}}
\sum_{j=0}^{M}
\Phi\big(\varsigma_{j}^{\zeta,\kappa,\mu}\big)\,
\mathfrak{P}_{i}^{\zeta,\kappa,\mu}\big(\varsigma_{j}^{\zeta,\kappa,\mu}\big)\,
w_{j}^{\zeta,\kappa,\mu}.
\end{equation}
\end{definition}

To measure approximation errors, we introduce an FJFs2-weighted Sobolev-type space.

\begin{definition}
For $p\in\mathbb{N}$, define
\begin{equation}
\mathcal{H}_{\zeta,\kappa}^{p,\mu}(\chi)
:=
\left\{
\Phi:
\mathfrak{D}_{\mu}^{(i)}\Phi
\in \mathcal{L}^{2}_{\varpi^{\zeta+i,\kappa+i,\mu}}(\chi),
\quad 0\le i\le p
\right\},
\end{equation}
where $\mathcal{L}^{2}_{\varpi^{\zeta+i,\kappa+i,\mu}}(\chi)$ denotes the weighted $\mathcal{L}^{2}$-space on $\chi$ with weight $\varpi^{\zeta+i,\kappa+i,\mu}$. The associated inner product, seminorm, and norm are defined by
\begin{equation}
\begin{split}
(\Psi,\Phi)_{\mathcal{H}_{\zeta,\kappa}^{p,\mu}}
&:=
\sum_{i=0}^{p}
\big(
\mathfrak{D}_{\mu}^{(i)}\Psi,
\mathfrak{D}_{\mu}^{(i)}\Phi
\big)_{\varpi^{\zeta+i,\kappa+i,\mu}},\\[3pt]
|\Phi|_{\mathcal{H}_{\zeta,\kappa}^{p,\mu}}
&:=
\big\|
\mathfrak{D}_{\mu}^{(p)}\Phi
\big\|_{\varpi^{\zeta+p,\kappa+p,\mu}},\\[3pt]
\|\Phi\|_{\mathcal{H}_{\zeta,\kappa}^{p,\mu}}
&:=
(\Phi,\Phi)_{\mathcal{H}_{\zeta,\kappa}^{p,\mu}}^{1/2}.
\end{split}
\end{equation}
\end{definition}

For $\mu=1$, this reduces to the classical Jacobi-weighted Sobolev space on $\chi$: \begin{equation} \mathcal{H}_{\zeta,\kappa}^{s,1}(\chi) := \left\{ \Phi: \partial_{\varsigma}^{\,i}\Phi \in \mathcal{L}^{2}_{\varpi^{\zeta+i,\kappa+i,1}}(\chi), \quad 0\le i\le s \right\}, \qquad s\in\mathbb{N}. \end{equation}
\textcolor{black}{Here, $\partial_{\varsigma}^{\,i}\Phi$ denotes the $i$th-order ordinary derivative of $\Phi$ with respect to the variable $\varsigma$.}

The next lemma relates the fractional space $\mathcal{H}_{\zeta,\kappa}^{r,\mu}(\chi)$ to its classical counterpart through the fractional change of variables.

\begin{lemma}
A function $\Phi \in \mathcal{H}_{\zeta,\kappa}^{r,\mu}(\chi)$ if and only if
\[
\widetilde{\Phi}(\varsigma)
:=
\Phi\big(1-(1-\varsigma)^{1/\mu}\big)
\in \mathcal{H}_{\zeta,\kappa}^{r,1}(\chi).
\]
\end{lemma}

\begin{proof}
Set
\[
\eta = 1-(1-\varsigma)^{\mu}, \qquad \varsigma\in(0,1).
\]
Then
\[
\varsigma = 1-(1-\eta)^{1/\mu},
\qquad
d\eta = \mu(1-\varsigma)^{\mu-1}d\varsigma.
\]
For $0\le s\le r$, we obtain
\begin{equation}
\begin{split}
\big\|\mathfrak{D}_{\mu}^{(s)}\Phi\big\|_{\varpi^{\zeta+s,\kappa+s,\mu}}^{2}
&=
\int_{0}^{1}
\varpi^{\zeta+s,\kappa+s,\mu}(\varsigma)\,
\big(\mathfrak{D}_{\mu}^{(s)}\Phi(\varsigma)\big)^{2}\,d\varsigma\\
&=
\int_{0}^{1}
\varpi^{\zeta+s,\kappa+s,1}(\eta)\,
\big(\partial_{\eta}^{\,s}\widetilde{\Phi}(\eta)\big)^{2}\,d\eta\\
&=
\big\|\partial_{\eta}^{\,s}\widetilde{\Phi}\big\|_{\varpi^{\zeta+s,\kappa+s,1}}^{2}.
\end{split}
\end{equation}
Hence, the norms are equivalent, and the result follows.
\end{proof}

\begin{lemma}(\cite{zaky2024non})\label{lemma41}
Let
\[
\widetilde{\Phi}(\varsigma):=\Phi\big(1-(1-\varsigma)^{1/\mu}\big)
\in \mathcal{H}_{\zeta,\kappa}^{j,1}(\chi),
\quad \textcolor{black}{1\le j\le M+1}.
\]
Then the interpolation error estimate
\begin{equation}\label{s77}
\big\|
\Phi-\mathcal{I}_{\varsigma,M}^{\zeta,\kappa,\mu}\Phi
\big\|_{\varpi^{\zeta,\kappa,\mu}}
\le
C\,M^{-j}
\big\|
\partial_{\varsigma}^{\,j}\widetilde{\Phi}
\big\|_{\varpi^{\zeta+j,\kappa+j,1}}
\end{equation}
holds, where $C>0$ is a constant independent of $M$ and $\Phi$.
\end{lemma}

\begin{definition}
Let $\mathbf{B}(\varsigma)=\big(b_{ij}(\varsigma)\big)_{m\times n}$ be a matrix-valued function on $\chi$. Its pointwise magnitude is defined by
\[
|\mathbf{B}(\varsigma)|
:=
\sum_{i=1}^{m}\sum_{j=1}^{n}|b_{ij}(\varsigma)|.
\]
For the weight $\varpi^{\zeta,\kappa,\mu}$, we define the associated weighted $\mathcal{L}^{2}$-norm by
\[
\|\mathbf{B}\|_{\varpi^{\zeta,\kappa,\mu}}
:=
\left(
\int_{\chi}
|\mathbf{B}(\varsigma)|^{2}\,
\varpi^{\zeta,\kappa,\mu}(\varsigma)\,d\varsigma
\right)^{1/2}.
\]
For vector-valued functions, this notation is understood componentwise.
\end{definition}

\begin{lemma}\label{lemma:matrix_interp}
Let
\[
\widetilde{\mathbf{V}}(\varsigma)
:=
\mathbf{V}\big(1-(1-\varsigma)^{1/\mu}\big)
\in \mathcal{H}_{\zeta,\kappa}^{s,1}(\chi),
\quad \textcolor{black}{1\le s\le M+1},
\]
where $\mathbf{V}=(V_1,\dots,V_Q)^{\top}$ is a vector-valued function. Then
\textcolor{black}{\begin{equation*}
\Big\|
\mathbf{V}-\mathcal{I}_{\varsigma,M}^{\zeta,\kappa,\mu}\mathbf{V}
\Big\|_{\varpi^{\zeta,\kappa,\mu}}
\le
C\,M^{-s}\,
\Big\|
\partial_{\varsigma}^{\,s}\widetilde{\mathbf{V}}
\Big\|_{\varpi^{\zeta+s,\kappa+s,1}},
\end{equation*}}
where $C>0$ is a generic constant independent of both $\mathbf{V}$ and $M$.
\end{lemma}

\begin{proof}
By the definition of the pointwise magnitude,
\[
\big|
\big(\mathbf{V}-\mathcal{I}_{\varsigma,M}^{\zeta,\kappa,\mu}\mathbf{V}\big)(\varsigma)
\big|
=
\sum_{q=1}^{Q}
\big|
V_q(\varsigma)-\mathcal{I}_{\varsigma,M}^{\zeta,\kappa,\mu}V_q(\varsigma)
\big|.
\]
Hence,
\begin{equation}\label{eq:step1}
\begin{aligned}
\Big\|
\mathbf{V}-\mathcal{I}_{\varsigma,M}^{\zeta,\kappa,\mu}\mathbf{V}
\Big\|_{\varpi^{\zeta,\kappa,\mu}}
&=
\left(
\int_{\chi}
\Big(
\sum_{q=1}^{Q}
\big|
V_q-\mathcal{I}_{\varsigma,M}^{\zeta,\kappa,\mu}V_q
\big|
\Big)^2
\varpi^{\zeta,\kappa,\mu}(\varsigma)\,d\varsigma
\right)^{1/2}.
\end{aligned}
\end{equation}
Applying the Cauchy--Schwarz inequality   yields
\begin{equation}\label{eq:step2}
\begin{aligned}
\Big\|
\mathbf{V}-\mathcal{I}_{\varsigma,M}^{\zeta,\kappa,\mu}\mathbf{V}
\Big\|_{\varpi^{\zeta,\kappa,\mu}}
&\le
C
\left(
\sum_{q=1}^{Q}
\Big\|
V_q-\mathcal{I}_{\varsigma,M}^{\zeta,\kappa,\mu}V_q
\Big\|_{\varpi^{\zeta,\kappa,\mu}}^{2}
\right)^{1/2},
\end{aligned}
\end{equation}
for some constant $C>0$ depending only on $Q$.

By Lemma~\ref{lemma41}, for each component $V_q$ we have
\begin{equation}\label{eq:scalar}
\Big\|
V_q-\mathcal{I}_{\varsigma,M}^{\zeta,\kappa,\mu}V_q
\Big\|_{\varpi^{\zeta,\kappa,\mu}}
\le
C\,M^{-s}
\Big\|
\partial_{\varsigma}^{\,s}\widetilde{V}_q
\Big\|_{\varpi^{\zeta+s,\kappa+s,1}},
\qquad q=1,\dots,Q,
\end{equation}
where
\[
\widetilde{V}_q(\varsigma)=V_q\big(1-(1-\varsigma)^{1/\mu}\big).
\]
Substituting \eqref{eq:scalar} into \eqref{eq:step2}, we obtain
\begin{equation}\label{eq:step3}
\begin{aligned}
\Big\|
\mathbf{V}-\mathcal{I}_{\varsigma,M}^{\zeta,\kappa,\mu}\mathbf{V}
\Big\|_{\varpi^{\zeta,\kappa,\mu}}
&\le
C\,M^{-s}
\left(
\sum_{q=1}^{Q}
\Big\|
\partial_{\varsigma}^{\,s}\widetilde{V}_q
\Big\|_{\varpi^{\zeta+s,\kappa+s,1}}^{2}
\right)^{1/2}.
\end{aligned}
\end{equation}

Finally, since
\[
\sum_{q=1}^{Q}
\big|
\partial_{\varsigma}^{\,s}\widetilde{V}_q(\varsigma)
\big|^{2}
\le
\left(
\sum_{q=1}^{Q}
\big|
\partial_{\varsigma}^{\,s}\widetilde{V}_q(\varsigma)
\big|
\right)^{2}
=
\big|
\partial_{\varsigma}^{\,s}\widetilde{\mathbf{V}}(\varsigma)
\big|^{2},
\]
we get
\[
\sum_{q=1}^{Q}
\Big\|
\partial_{\varsigma}^{\,s}\widetilde{V}_q
\Big\|_{\varpi^{\zeta+s,\kappa+s,1}}^{2}
\le
\Big\|
\partial_{\varsigma}^{\,s}\widetilde{\mathbf{V}}
\Big\|_{\varpi^{\zeta+s,\kappa+s,1}}^{2}.
\]
Combining this estimate with \eqref{eq:step3} completes the proof.
\end{proof}

\section{The FJFs2--Gauss collocation method}\label{sec3}

Rewriting each equation in system~\eqref{eq1.1} through the integral representation of the right-sided fractional integral operator yields the equivalent formulation
\textcolor{black}{\begin{equation}\label{1.2vvva}
\begin{aligned}
U_r(z)
&= \frac{1}{\Gamma(\beta)}
\int_{z}^{1}
\mu\bigl((1-z)^{\mu}-(1-\varsigma)^{\mu}\bigr)^{\beta-1}
(1-z)^{\beta-1}\,K_{\mu}(z,\varsigma)\,
F_r\!\big(\varsigma,U_1(\varsigma),U_2(\varsigma),\ldots,U_Q(\varsigma)\big)\, d\varsigma,\\[4pt]
K_{\mu}(z,\varsigma)
&=
\frac{1}{\mu}\,
(1-\varsigma)^{1-\mu}
\left(
\frac{\varsigma-z}
{(1-z)^{\mu}-(1-\varsigma)^{\mu}}
\right)^{\beta-1},
\qquad 0< z < \varsigma < 1,\quad r = 1,2,\ldots,Q.
\end{aligned}
\end{equation}}
Introducing the transformation
\[
\varsigma=\varsigma(z,\eta):=z+(1-z)\eta,
\]
the integral formulation \eqref{1.2vvva} can be rewritten, for each component, as
\begin{equation}\label{2.9a}
\begin{aligned}
U_r(z)
&= \frac{(1-z)^{\beta \mu}}{\Gamma(\beta)}
\int_{0}^{1}
\varpi^{0,\beta-1,\mu}(\eta)\,
\widehat{F_r}\!\left(
z,\varsigma(z,\eta),U_1(\varsigma(z,\eta)),\ldots,
U_Q(\varsigma(z,\eta))
\right)\,d\eta,
\end{aligned}
\end{equation}
where
\begin{equation}\label{eq:def-hatG-right}
\begin{aligned}
&\widehat{F_r}\!\left(
z,\varsigma(z,\eta),U_1(\varsigma(z,\eta)),\ldots,
U_Q(\varsigma(z,\eta))
\right)
=
K_{\mu}(z,\varsigma)\,
F_r\!\big(\varsigma,U_1(\varsigma),U_2(\varsigma),\ldots,U_Q(\varsigma)\big),
\\[-2pt]
&\hspace{4cm}
r=1,2,\ldots,Q,\qquad \varsigma=\varsigma(z,\eta)=z+(1-z)\eta,
\end{aligned}
\end{equation}
and the FJFs2 weight is given by
\[
\varpi^{0,\beta-1,\mu}(\eta)
=
\mu\bigl(1-(1-\eta)^{\mu}\bigr)^{\beta-1}(1-\eta)^{\mu-1}.
\]

The FJFs2--Gauss pseudospectral approximation of \eqref{2.9a} seeks an approximate solution $U_{r,M}$ whose components are expanded in the FJFs2 basis introduced in Section~\ref{sec2}. For each component, we set
\begin{equation}\label{3.8aa}
\begin{aligned}
U_{r,M}(z)
&= \sum_{k=0}^{M}
\widehat{u}_{r,k}\,
\mathfrak{P}_{k}^{0,\beta-1,\mu}(z),
\qquad
U_{r,M}(z)\in \mathbb{S}_{M}^{\mu}(\chi),
\\[-2pt]
&\hspace{8cm} r = 1,2,\ldots,Q.
\end{aligned}
\end{equation}

Substituting the approximation \eqref{3.8aa} into \eqref{2.9a} and applying the FJFs2--Gauss interpolation operators
$\mathcal{I}_{z,M}^{0,\beta-1,\mu}$ and
$\mathcal{I}_{\eta,M}^{0,\beta-1,\mu}$, we obtain, for each $r=1,2,\ldots,Q$,
\begin{equation}\label{3.7a}
\begin{aligned}
U_{r,M}(z)
&= \frac{1}{\Gamma(\beta)}\,
\mathcal{I}_{z,M}^{0,\beta-1,\mu}
\Bigg[
(1-z)^{\beta \mu}
\int_{0}^{1}
\varpi^{0,\beta-1,\mu}(\eta)\,
\mathcal{I}_{\eta,M}^{0,\beta-1,\mu}
\\[-2pt]
&\qquad\qquad\times
\widehat{F_r}\!\left(
z,\varsigma(z,\eta),
U_{1,M}(\varsigma(z,\eta)),
U_{2,M}(\varsigma(z,\eta)),
\ldots,
U_{Q,M}(\varsigma(z,\eta))
\right)\,d\eta
\Bigg],
\\[-2pt]
&\hspace{8cm} r = 1,2,\ldots,Q.
\end{aligned}
\end{equation}
Here, all interpolations are performed in the FJFs2 space $\mathbb{S}_{M}^{\mu}(\chi)$.

To implement \eqref{3.7a}, we expand the doubly interpolated term in the tensor-product FJFs2 basis. For each component $m=1,2,\ldots,Q$, define
\begin{equation}\label{3.8a}
\begin{aligned}
&\mathcal{I}_{z,M}^{0,\beta-1,\mu}
\mathcal{I}_{\eta,M}^{0,\beta-1,\mu}
\Big[
(1-z)^{\beta \mu}\,
\widehat{F_m}\!\left(
z,\varsigma(z,\eta),
U_{1,M}(\varsigma(z,\eta)),\ldots,U_{Q,M}(\varsigma(z,\eta))
\right)
\Big]
\\[-2pt]
&\qquad=
\sum_{i=0}^{M}\sum_{j=0}^{M}
\widehat{v}_{m,i,j}\,
\mathfrak{P}_{i}^{0,\beta-1,\mu}(z)\,
\mathfrak{P}_{j}^{0,\beta-1,\mu}(\eta),
\qquad m=1,2,\ldots,Q.
\end{aligned}
\end{equation}

Using the representation \eqref{3.8a}, the integral term in \eqref{3.7a} can be written as
\begin{equation}\label{3.9a}
\begin{aligned}
&\frac{1}{\Gamma(\beta)}
\int_{0}^{1}
\varpi^{0,\beta-1,\mu}(\eta)\,
\mathcal{I}_{z,M}^{0,\beta-1,\mu}
\mathcal{I}_{\eta,M}^{0,\beta-1,\mu}
\\[-2pt]
&\hspace{1.8cm}\times
\Big[
(1-z)^{\beta \mu}
\widehat{F_m}\!\left(
z,\varsigma(z,\eta),
U_{1,M}(\varsigma(z,\eta)),\ldots,
U_{Q,M}(\varsigma(z,\eta))
\right)
\Big]\,d\eta
\\[-2pt]
&\quad=
\frac{1}{\Gamma(\beta)}
\sum_{i=0}^{M}\sum_{j=0}^{M}
\widehat{v}_{m,i,j}\,
\mathfrak{P}_{i}^{0,\beta-1,\mu}(z)
\int_{0}^{1}
\varpi^{0,\beta-1,\mu}(\eta)\,
\mathfrak{P}_{j}^{0,\beta-1,\mu}(\eta)\,d\eta
\\[-2pt]
&\quad=
\frac{1}{\Gamma(\beta+1)}
\sum_{i=0}^{M}
\widehat{v}_{m,i,0}\,
\mathfrak{P}_{i}^{0,\beta-1,\mu}(z),
\qquad m=1,2,\ldots,Q.
\end{aligned}
\end{equation}

Applying \eqref{3.4a} together with \eqref{3.8a}, the coefficients
$\widehat{v}_{m,i,0}$ admit the explicit representation
\textcolor{black}{\begin{equation}\label{3.10a}
\begin{aligned}
\widehat{v}_{m,i,0}
&=
\beta(2i+\beta)
\sum_{r=0}^{M}\sum_{s=0}^{M}
\bigl(1-z_{r}^{0,\beta-1,\mu}\bigr)^{\beta \mu}
\\
&\quad\times
\widehat{F_m}\!\Big(
z_{r}^{0,\beta-1,\mu},
\varsigma\big(z_{r}^{0,\beta-1,\mu},\eta_{s}^{0,\beta-1,\mu}\big),
U_{1,M}\big(\varsigma(z_{r}^{0,\beta-1,\mu},\eta_{s}^{0,\beta-1,\mu})\big),
\ldots,
U_{Q,M}\big(\varsigma(z_{r}^{0,\beta-1,\mu},\eta_{s}^{0,\beta-1,\mu})\big)
\Big)
\\
&\quad\times
\mathfrak{P}_{i}^{0,\beta-1,\mu}
\big(z_{r}^{0,\beta-1,\mu}\big)\,
w_{r}^{0,\beta-1,\mu}\,
w_{s}^{0,\beta-1,\mu}\textcolor{black}{.}
\end{aligned}
\end{equation}}
Here,
$\{z_{r}^{0,\beta-1,\mu},\,w_{r}^{0,\beta-1,\mu}\}_{r=0}^{M}$
and
$\{\eta_{s}^{0,\beta-1,\mu},\,w_{s}^{0,\beta-1,\mu}\}_{s=0}^{M}$
represent the FJFs2--Gauss quadrature nodes and corresponding weights associated with the weight function
$\varpi^{0,\beta-1,\mu}$.

Combining \eqref{3.8aa}--\eqref{3.10a}, we obtain
\begin{equation}\label{3.13a}
\sum_{i=0}^{M}\widehat{u}_{m,i}\,
\mathfrak{P}_{i}^{0,\beta-1,\mu}(z)
=
\frac{1}{\Gamma(\beta+1)}
\sum_{i=0}^{M}\widehat{v}_{m,i,0}\,
\mathfrak{P}_{i}^{0,\beta-1,\mu}(z),
\qquad m=1,2,\ldots,Q.
\end{equation}
By comparing coefficients, or equivalently by using orthogonality, we deduce
\begin{equation}\label{5-27}
\widehat{u}_{m,j}
=
\frac{1}{\Gamma(\beta+1)}\,\widehat{v}_{m,j,0},
\qquad j=0,1,\ldots,M,\quad m=1,2,\ldots,Q.
\end{equation}

The nonlinear algebraic system \eqref{5-27}, which determines the unknown coefficients $\{\widehat{u}_{m,j}\}$, can be solved efficiently by standard iterative techniques, such as Newton-type methods. In the numerical experiments, the resulting system is solved in \textsc{Mathematica} using the \texttt{FindRoot} command with zero initial guesses for all coefficients.
\section{Convergence analysis}\label{sec:conv}

We now proceed to the convergence analysis of the proposed method. To this end, we introduce the fractional Jacobi--Gauss interpolation operator in the $\varsigma$-direction. For fixed $z\in(0,1)$ and $z<\varsigma\le 1$, let
\[
{}_z\widetilde{\mathcal{I}}_{\varsigma,M}^{0,\beta-1,\mu} v(\varsigma)
:=
\left.
\mathcal{I}_{\eta,M}^{0,\beta-1,\mu}
\big[v(\varsigma(z,\eta))\big]
\right|_{\eta=\frac{\varsigma-z}{1-z}},
\qquad
\varsigma=\varsigma(z,\eta):=z+(1-z)\eta,\quad 0<\eta<1.
\]

For any $v\in C([z,1])$, the function
${}_z\widetilde{\mathcal{I}}_{\varsigma,M}^{0,\beta-1,\mu} v$
belongs to the space of polynomials of degree at most $M$ in the stretched variable
$\eta=(\varsigma-z)/(1-z)$ and, by definition, satisfies the interpolation condition
\[
{}_z\widetilde{\mathcal{I}}_{\varsigma,M}^{0,\beta-1,\mu}
v\bigl(\varsigma_j^{0,\beta-1,\mu}(z)\bigr)
=
v\bigl(\varsigma_j^{0,\beta-1,\mu}(z)\bigr),
\qquad
\varsigma_j^{0,\beta-1,\mu}(z)
= \varsigma\bigl(z,\eta_j^{0,\beta-1,\mu}\bigr),
\]
where $\{\eta_j^{0,\beta-1,\mu}\}_{j=0}^M$ denote the fractional
Jacobi--Gauss nodes corresponding to the weight
\[
\varpi^{0,\beta-1,\mu}(\eta)
=
\mu\bigl(1-(1-\eta)^{\mu}\bigr)^{\beta-1}(1-\eta)^{\mu-1}.
\]

\textcolor{black}{Using the fractional Jacobi--Gauss quadrature rule in the $\eta$-variable along with the mapping $\varsigma=\varsigma(z,\eta)$, we have
\begin{equation}\label{eq:A2-right}
\begin{aligned}
&\int_{z}^{1}
\mu\bigl((1-z)^{\mu}-(1-\varsigma)^{\mu}\bigr)^{\beta-1}
(1-\varsigma)^{\mu-1}\,
{}_z\widetilde{\mathcal{I}}_{\varsigma,M}^{0,\beta-1,\mu}
v(\varsigma)\,d\varsigma \\
&\qquad =
(1-z)^{\beta\mu}
\int_{0}^{1}
\varpi^{0,\beta-1,\mu}(\eta)\,
\mathcal{I}_{\eta,M}^{0,\beta-1,\mu}
\!\big[v(\varsigma(z,\eta))\big]\,d\eta \\
&\qquad =
(1-z)^{\beta\mu}
\sum_{j=0}^{M}
v\bigl(\varsigma_j^{0,\beta-1,\mu}(z)\bigr)\,
\omega_j^{0,\beta-1,\mu},
\end{aligned}
\end{equation}
where $\{\omega_j^{0,\beta-1,\mu}\}_{j=0}^M$ are the associated
quadrature weights.}

\textcolor{black}{In an analogous manner, we derive the discrete identity
\begin{equation}\label{eq:A4-right}
\int_{z}^{1}
\mu\bigl((1-z)^{\mu}-(1-\varsigma)^{\mu}\bigr)^{\beta-1}
(1-\varsigma)^{\mu-1}
\Big(
{}_z\widetilde{\mathcal{I}}_{\varsigma,M}^{0,\beta-1,\mu} v(\varsigma)
\Big)^{2}
\,d\varsigma
=
(1-z)^{\beta\mu}
\sum_{j=0}^{M}
v\bigl(\varsigma_j^{0,\beta-1,\mu}(z)\bigr)^{2}\,
\omega_j^{0,\beta-1,\mu}.
\end{equation}}

\textcolor{black}{We now establish an error bound for the interpolation operator in the $\varsigma$-direction.
By Lemma~\ref{lemma41}, for every integer $1\le m\le M+1$, we have
\begin{equation}\label{eq:frac-4.5-like-right}
\begin{aligned}
\int_{z}^{1}
\mu\bigl((1-z)^{\mu}&-(1-\varsigma)^{\mu}\bigr)^{\beta-1}
(1-\varsigma)^{\mu-1}
\bigl|
v(\varsigma)
-
{}_z\widetilde{\mathcal{I}}_{\varsigma,M}^{0,\beta-1,\mu}v(\varsigma)
\bigr|^{2}
\,d\varsigma\\
&=
(1-z)^{\beta\mu}
\int_{0}^{1}
\varpi^{0,\beta-1,\mu}(\eta)
\bigl|
v(\varsigma(z,\eta))
-
\mathcal{I}_{\eta,M}^{0,\beta-1,\mu}
v(\varsigma(z,\eta))
\bigr|^{2}
\,d\eta
\\[1mm]
&\le
c\,M^{-2m}\,(1-z)^{\beta\mu}
\int_{0}^{1}
\varpi^{m,\beta-1+m,1}(\eta)
\bigl|
\partial_{\eta}^{\,m}
v\bigl(\varsigma(z,1-(1-\eta)^{1/\mu})\bigr)
\bigr|^{2}
\,d\eta
\\[1mm]
&\le
c\,M^{-2m}
\int_{(1-z)^{\mu}}^1
\bigl((1-z)^{\mu}-(1-\varsigma)\bigr)^{\beta+m-1}(1-\varsigma)^{m}
\bigl|
\partial_{\varsigma}^{\,m}
v\!\bigl(1-(1-\varsigma)^{1/\mu}\bigr)
\bigr|^{2}
\,d\varsigma,
\end{aligned}
\end{equation}
where the constant $c>0$ is independent of $M$ and $z$.}
\medskip

To study the convergence of the proposed FJFs2--Gauss pseudospectral scheme,
we first rewrite \eqref{2.9a} in vector form as
\begin{equation}\label{2.9b}
\mathbf{U}(z)
=
\frac{(1-z)^{\beta\mu}}{\Gamma(\beta)}
\int_{0}^{1}
\varpi^{0,\beta-1,\mu}(\eta)\,
\widehat{\mathbf{F}}\!\Big(
z,\varsigma(z,\eta),\mathbf{U}(\varsigma(z,\eta))
\Big)\,d\eta,
\qquad z\in(0,1),
\end{equation}
where
\[
\mathbf{U}(z):=
\begin{bmatrix}
U_1(z)\\[2pt]
U_2(z)\\[2pt]
\vdots\\[2pt]
U_Q(z)
\end{bmatrix},
\]
and
\[
\widehat{\mathbf{F}}\!\Big(
z,\varsigma(z,\eta),\mathbf{U}(\varsigma(z,\eta))
\Big)
:=
\begin{bmatrix}
\widehat{F_1}\!\left(z,\varsigma(z,\eta),U_1(\varsigma(z,\eta)),\ldots,U_Q(\varsigma(z,\eta))\right)\\[2pt]
\widehat{F_2}\!\left(z,\varsigma(z,\eta),U_1(\varsigma(z,\eta)),\ldots,U_Q(\varsigma(z,\eta))\right)\\[2pt]
\vdots\\[2pt]
\widehat{F_Q}\!\left(z,\varsigma(z,\eta),U_1(\varsigma(z,\eta)),\ldots,U_Q(\varsigma(z,\eta))\right)
\end{bmatrix}.
\]

Let $\mathbf{U}(z)$ denote the exact solution of \eqref{2.9b}, and let
$\mathbf{U}_M(z)$ be the numerical approximation obtained from the discrete
scheme \eqref{3.7a}. We introduce the error vector
\[
\mathbf{E}_M(z):=\mathbf{U}(z)-\mathbf{U}_M(z),
\]
and measure it in the weighted $\mathcal{L}^{2}$-norm associated with
$\varpi^{0,\beta-1,\mu}$. By the triangle inequality,
\begin{equation}\label{5-29}
\big\| \mathbf{E}_M \big\|_{\varpi^{0,\beta-1,\mu}}
\le
\big\| \mathbf{U}-\mathcal{I}_{z,M}^{0,\beta-1,\mu}\mathbf{U} \big\|_{\varpi^{0,\beta-1,\mu}}
+
\big\| \mathcal{I}_{z,M}^{0,\beta-1,\mu}\mathbf{U}-\mathbf{U}_M \big\|_{\varpi^{0,\beta-1,\mu}}.
\end{equation}

The second term on the right-hand side of \eqref{5-29} is decomposed into
three contributions corresponding to interpolation in the $z$-variable,
quadrature/interpolation in the auxiliary variable $\eta$, and the nonlinearity
of $\widehat{\mathbf{F}}$.

\begin{lemma}
The error $\mathbf{E}_M$ admits the decomposition
\begin{equation}\label{e12}
\big\| \mathbf{E}_M \big\|_{\varpi^{0,\beta-1,\mu}}
\le
\sum_{i=1}^{3}
\big\| \mathbf{E}_i \big\|_{\varpi^{0,\beta-1,\mu}},
\end{equation}
where
\begin{equation}\label{e13}
\begin{split}
\mathbf{E}_1(z)
&=
\mathbf{U}(z)-\mathcal{I}_{z,M}^{0,\beta-1,\mu}\mathbf{U}(z),
\\[4pt]
\mathbf{E}_2(z)
&=
\frac{1}{\Gamma(\beta)}\,
\mathcal{I}_{z,M}^{0,\beta-1,\mu}
\Bigg[
(1-z)^{\beta\mu}
\int_{0}^{1}
\varpi^{0,\beta-1,\mu}(\eta)\,
\\[-2pt]
&\qquad\qquad\qquad\times
\bigl(\mathcal{I}-\mathcal{I}_{\eta,M}^{0,\beta-1,\mu}\bigr)
\widehat{\mathbf{F}}\!\Big(
z,\varsigma(z,\eta),\mathbf{U}(\varsigma(z,\eta))
\Big)\,d\eta
\Bigg],
\\[4pt]
\mathbf{E}_3(z)
&=
\frac{1}{\Gamma(\beta)}\,
\mathcal{I}_{z,M}^{0,\beta-1,\mu}
\Bigg[
(1-z)^{\beta\mu}
\int_{0}^{1}
\varpi^{0,\beta-1,\mu}(\eta)\,
\\[-2pt]
&\qquad\qquad\qquad\times
\mathcal{I}_{\eta,M}^{0,\beta-1,\mu}
\Big(
\widehat{\mathbf{F}}\!\Big(
z,\varsigma(z,\eta),\mathbf{U}(\varsigma(z,\eta))
\Big)
-
\widehat{\mathbf{F}}\!\Big(
z,\varsigma(z,\eta),\mathbf{U}_M(\varsigma(z,\eta))
\Big)
\Big)\,d\eta
\Bigg].
\end{split}
\end{equation}
Here, $\varsigma(z,\eta)=z+(1-z)\eta$, and $\mathcal{I}$ denotes the identity operator
on the corresponding vector-valued function space.
\end{lemma}

\begin{theorem}\label{the22}
Let $\mathbf{U}(z)$ be the exact solution of \eqref{2.9b}, and let
$\mathbf{U}_M(z)$ denote the numerical solution produced by \eqref{3.7a}.
Assume that $
\mathbf{U}\in \mathcal{H}^{\,l,\mu}_{0,\beta-1}(\chi),
\
1\le l\le M+1,$ 
and that $\widehat{\mathbf{F}}$ is Lipschitz continuous with respect to its last
$Q$ arguments, uniformly in the first two arguments, with Lipschitz constant
$L$ satisfying
$ L<\Gamma(\beta+1).$ 
Then the error satisfies
\begin{equation}\label{conv-rate-main}
\begin{split}
\|\mathbf{E}_M\|_{\varpi^{0,\beta-1,\mu}}
\le
C\,M^{-l}&
\Bigg(
\left\|
\partial_{z}^{\,l}
\mathbf{U}\bigl(1-(1-z)^{1/\mu}\bigr)
\right\|_{\varpi^{l,\beta-1+l,1}}
\\
&
+
\max_{0\le j\le M}
\left\|
\partial_{\varsigma}^{\,l}
\widehat{\mathbf{F}}
\Bigl(
z_j,\,
1-(1-\varsigma)^{1/\mu},\,
\mathbf{U}\bigl(1-(1-\varsigma)^{1/\mu}\bigr)
\Bigr)
\right\|_{\varpi^{l,\beta-1+l,1}}
\Bigg),
\end{split}
\end{equation}
where $z_j=z_j^{0,\beta-1,\mu}$ and the constant $C>0$ is independent of
$M$.
\end{theorem}

\begin{proof}
We estimate the three terms in the decomposition
\[
\|\mathbf{E}_M\|_{\varpi^{0,\beta-1,\mu}}
\le
\|\mathbf{E}_1\|_{\varpi^{0,\beta-1,\mu}}
+
\|\mathbf{E}_2\|_{\varpi^{0,\beta-1,\mu}}
+
\|\mathbf{E}_3\|_{\varpi^{0,\beta-1,\mu}}.
\]

We begin with the interpolation term. By
Lemma~\ref{lemma:matrix_interp}, we have
\begin{equation}\label{eq:E1-correct}
\begin{split}
\|\mathbf{E}_1\|_{\varpi^{0,\beta-1,\mu}}
&=
\big\|
\mathbf{U}
-
\mathcal{I}_{z,M}^{0,\beta-1,\mu}\mathbf{U}
\big\|_{\varpi^{0,\beta-1,\mu}}
\\
&\textcolor{black}{\le
C\,M^{-l}
\left\|
\partial_{z}^{\,l}
\mathbf{U}\bigl(1-(1-z)^{1/\mu}\bigr)
\right\|_{z^{l,\beta-1+l,1}},
\qquad
1\le l\le M+1.}
\end{split}
\end{equation}

We next estimate the second contribution $\mathbf{E}_2$. By definition,
\begin{equation}\label{E2-def-corrected}
\begin{split}
\mathbf{E}_2(z)
&=
\frac{1}{\Gamma(\beta)}
\mathcal{I}_{z,M}^{0,\beta-1,\mu}
\Bigg[
(1-z)^{\beta\mu}
\int_0^1
\varpi^{0,\beta-1,\mu}(\eta)\,
\\
&\qquad\qquad\qquad\times
\bigl(
\mathcal{I}-\mathcal{I}_{\eta,M}^{0,\beta-1,\mu}
\bigr)
\widehat{\mathbf{F}}
\bigl(
z,\varsigma(z,\eta),\mathbf{U}(\varsigma(z,\eta))
\bigr)\,d\eta
\Bigg].
\end{split}
\end{equation}

For each quadrature node $z_j:=z_j^{0,\beta-1,\mu}$, let
\[
\widehat{\mathbf{G}}_j(\varsigma)
:=
\widehat{\mathbf{F}}
\bigl(
z_j,\varsigma,\mathbf{U}(\varsigma)
\bigr),
\qquad
\varsigma\in[z_j,1].
\]
Then
\[
\widehat{\mathbf{G}}_j\bigl(\varsigma(z_j,\eta)\bigr)
=
\widehat{\mathbf{F}}
\bigl(
z_j,\varsigma(z_j,\eta),\mathbf{U}(\varsigma(z_j,\eta))
\bigr).
\]

Using the FJFs2--Gauss quadrature formula \eqref{3.3a} in the $z$-variable,
we obtain
\begin{equation}\label{E2-step1}
\begin{split}
\|\mathbf{E}_2\|_{\varpi^{0,\beta-1,\mu}}
&=
\frac{1}{\Gamma(\beta)}
\Bigg(
\sum_{j=0}^{M}
\omega_j^{0,\beta-1,\mu}
\Bigg|
(1-z_j)^{\beta\mu}
\int_0^1
\varpi^{0,\beta-1,\mu}(\eta)\,
\\
&\qquad\qquad\qquad\qquad\times
\bigl(
\mathcal{I}-\mathcal{I}_{\eta,M}^{0,\beta-1,\mu}
\bigr)
\widehat{\mathbf{F}}
\bigl(
z_j,\varsigma(z_j,\eta),\mathbf{U}(\varsigma(z_j,\eta))
\bigr)\,d\eta
\Bigg|^2
\Bigg)^{1/2}.
\end{split}
\end{equation}
Using the Cauchy--Schwarz inequality, we obtain
\begin{equation}\label{E2-CS}
\begin{split}
\Bigg|
\int_0^1
\varpi^{0,\beta-1,\mu}(\eta)\,
\mathbf{g}(\eta)\,d\eta
\Bigg|^2
&\le
\Bigg(
\int_0^1
\varpi^{0,\beta-1,\mu}(\eta)\,d\eta
\Bigg)
\Bigg(
\int_0^1
\varpi^{0,\beta-1,\mu}(\eta)\,
|\mathbf{g}(\eta)|^2\,d\eta
\Bigg)
\\
&=
\frac{1}{\beta}
\int_0^1
\varpi^{0,\beta-1,\mu}(\eta)\,
|\mathbf{g}(\eta)|^2\,d\eta,
\end{split}
\end{equation}
since
\[
\int_0^1 \varpi^{0,\beta-1,\mu}(\eta)\,d\eta=\frac{1}{\beta}.
\]

Substituting \eqref{E2-CS} into \eqref{E2-step1} yields
\begin{equation}\label{E2-step2}
\begin{split}
\|\mathbf{E}_2\|_{\varpi^{0,\beta-1,\mu}}
&\le
\frac{1}{\Gamma(\beta)}
\Bigg(
\sum_{j=0}^{M}
\omega_j^{0,\beta-1,\mu}\,
(1-z_j)^{2\beta\mu}\,
\frac{1}{\beta}
\\
&\qquad\qquad\times
\int_0^1
\varpi^{0,\beta-1,\mu}(\eta)\,
\Big|
\bigl(
\mathcal{I}-\mathcal{I}_{\eta,M}^{0,\beta-1,\mu}
\bigr)
\widehat{\mathbf{F}}
\bigl(
z_j,\varsigma(z_j,\eta),\mathbf{U}(\varsigma(z_j,\eta))
\bigr)
\Big|^2\,d\eta
\Bigg)^{1/2}
\\
&=
\frac{1}{\Gamma(\beta+1)}
\Bigg(
\sum_{j=0}^{M}
\omega_j^{0,\beta-1,\mu}\,
\beta(1-z_j)^{\beta\mu}
\\
&\qquad\qquad\times
\Bigg[
(1-z_j)^{\beta\mu}
\int_0^1
\varpi^{0,\beta-1,\mu}(\eta)\,
\Big|
\bigl(
\mathcal{I}-\mathcal{I}_{\eta,M}^{0,\beta-1,\mu}
\bigr)
\widehat{\mathbf{F}}
\bigl(
z_j,\varsigma(z_j,\eta),\mathbf{U}(\varsigma(z_j,\eta))
\bigr)
\Big|^2\,d\eta
\Bigg]
\Bigg)^{1/2}.
\end{split}
\end{equation}

Therefore,
\begin{equation}\label{E2-step3}
\begin{split}
\|\mathbf{E}_2\|_{\varpi^{0,\beta-1,\mu}}
&\le
\frac{1}{\Gamma(\beta+1)}
\Bigg(
\sum_{j=0}^{M}
\omega_j^{0,\beta-1,\mu}\,
\beta(1-z_j)^{\beta\mu}
\Bigg)^{1/2}
\\
&\qquad\times
\max_{0\le j\le M}
\Bigg[
(1-z_j)^{\beta\mu}
\int_0^1
\varpi^{0,\beta-1,\mu}(\eta)\,
\Big|
\bigl(
\mathcal{I}-\mathcal{I}_{\eta,M}^{0,\beta-1,\mu}
\bigr)
\widehat{\mathbf{F}}
\bigl(
z_j,\varsigma(z_j,\eta),\mathbf{U}(\varsigma(z_j,\eta))
\bigr)
\Big|^2\,d\eta
\Bigg]^{1/2}.
\end{split}
\end{equation}

Since $0\le (1-z_j)^{\beta\mu}\le 1$ for all $j$, while the quadrature
weights are positive and exact for constants, we have
\begin{equation}\label{E2-weight-bound}
\sum_{j=0}^{M}
\omega_j^{0,\beta-1,\mu}\,
\beta(1-z_j)^{\beta\mu}
\le
\beta\sum_{j=0}^{M}\omega_j^{0,\beta-1,\mu}
=
1.
\end{equation}
Hence \eqref{E2-step3} reduces to
\begin{equation}\label{E2-step4}
\begin{split}
\|\mathbf{E}_2\|_{\varpi^{0,\beta-1,\mu}}
&\le
\frac{1}{\Gamma(\beta+1)}
\max_{0\le j\le M}
\Bigg[
(1-z_j)^{\beta\mu}
\int_0^1
\varpi^{0,\beta-1,\mu}(\eta)\,
\\
&\qquad\qquad\times
\Big|
\bigl(
\mathcal{I}-\mathcal{I}_{\eta,M}^{0,\beta-1,\mu}
\bigr)
\widehat{\mathbf{F}}
\bigl(
z_j,\varsigma(z_j,\eta),\mathbf{U}(\varsigma(z_j,\eta))
\bigr)
\Big|^2\,d\eta
\Bigg]^{1/2}.
\end{split}
\end{equation}

We now apply the interpolation estimate \eqref{eq:frac-4.5-like-right} with
\[
v(\varsigma)=\widehat{\mathbf{G}}_j(\varsigma)
=
\widehat{\mathbf{F}}
\bigl(
z_j,\varsigma,\mathbf{U}(\varsigma)
\bigr).
\]
Then
\begin{equation}\label{E2-step5}
\begin{aligned}
&(1-z_j)^{\beta\mu}
\int_0^1
\varpi^{0,\beta-1,\mu}(\eta)\,
\Big|
\bigl(
\mathcal{I}-\mathcal{I}_{\eta,M}^{0,\beta-1,\mu}
\bigr)
\widehat{\mathbf{F}}
\bigl(
z_j,\varsigma(z_j,\eta),\mathbf{U}(\varsigma(z_j,\eta))
\bigr)
\Big|^2\,d\eta
\\
&\qquad=
\int_{z_j}^{1}
\mu
\bigl((1-z_j)^{\mu}-(1-\varsigma)^{\mu}\bigr)^{\beta-1}
(1-\varsigma)^{\mu-1}
\Big|
\widehat{\mathbf{G}}_j(\varsigma)
-
{}_{z_j}\widetilde{\mathcal{I}}_{\varsigma,M}^{0,\beta-1,\mu}
\widehat{\mathbf{G}}_j(\varsigma)
\Big|^2\,d\varsigma
\\
&\qquad\le
c\,M^{-2l}
\int_{(1-z_j)^{\mu}}^{1}
\bigl((1-z_j)^{\mu}-(1-\varsigma)\bigr)^{\beta+l-1}
(1-\varsigma)^l
\\
&\qquad\qquad\qquad\times
\Big|
\partial_{\varsigma}^{\,l}
\widehat{\mathbf{F}}
\Bigl(
z_j,\,
1-(1-\varsigma)^{1/\mu},\,
\mathbf{U}\bigl(1-(1-\varsigma)^{1/\mu}\bigr)
\Bigr)
\Big|^2\,d\varsigma.
\end{aligned}
\end{equation}

Substituting \eqref{E2-step5} into \eqref{E2-step4}, we obtain
\begin{equation}\label{E2-final-detailed}
\begin{aligned}
\|\mathbf{E}_2\|_{\varpi^{0,\beta-1,\mu}}
&\le
\frac{c\,M^{-l}}{\Gamma(\beta+1)}
\max_{0\le j\le M}
\Bigg(
\int_{(1-z_j)^{\mu}}^{1}
\bigl((1-z_j)^{\mu}-(1-\varsigma)\bigr)^{\beta+l-1}
(1-\varsigma)^l
\\
&\qquad\qquad\qquad\times
\Big|
\partial_{\varsigma}^{\,l}
\widehat{\mathbf{F}}
\Bigl(
z_j,\,
1-(1-\varsigma)^{1/\mu},\,
\mathbf{U}\bigl(1-(1-\varsigma)^{1/\mu}\bigr)
\Bigr)
\Big|^2\,d\varsigma
\Bigg)^{1/2}
\\
&\le
C\,M^{-l}
\max_{0\le j\le M}
\left\|
\partial_{\varsigma}^{\,l}
\widehat{\mathbf{F}}
\Bigl(
z_j,\,
1-(1-\varsigma)^{1/\mu},\,
\mathbf{U}\bigl(1-(1-\varsigma)^{1/\mu}\bigr)
\Bigr)
\right\|_{\varpi^{l,\beta-1+l,1}}.
\end{aligned}
\end{equation}

We now estimate the third contribution $\mathbf{E}_3$. By definition,
\begin{equation}\label{E3-def-corrected}
\begin{split}
\mathbf{E}_3(z)
&=
\frac{1}{\Gamma(\beta)}
\mathcal{I}_{z,M}^{0,\beta-1,\mu}
\Bigg[
(1-z)^{\beta\mu}
\int_0^1
\varpi^{0,\beta-1,\mu}(\eta)\,
\\
&\qquad\qquad\qquad\times
\mathcal{I}_{\eta,M}^{0,\beta-1,\mu}
\Delta\widehat{\mathbf{F}}(z,\eta)\,d\eta
\Bigg],
\end{split}
\end{equation}
where
\[
\Delta\widehat{\mathbf{F}}(z,\eta)
:=
\widehat{\mathbf{F}}
\bigl(
z,\varsigma(z,\eta),\mathbf{U}(\varsigma(z,\eta))
\bigr)
-
\widehat{\mathbf{F}}
\bigl(
z,\varsigma(z,\eta),\mathbf{U}_M(\varsigma(z,\eta))
\bigr).
\]

For each quadrature node $z_j:=z_j^{0,\beta-1,\mu}$, let
\[
\Delta\widehat{\mathbf{G}}_j(\varsigma)
:=
\widehat{\mathbf{F}}
\bigl(
z_j,\varsigma,\mathbf{U}(\varsigma)
\bigr)
-
\widehat{\mathbf{F}}
\bigl(
z_j,\varsigma,\mathbf{U}_M(\varsigma)
\bigr),
\qquad
\varsigma\in[z_j,1].
\]
Then
\[
\Delta\widehat{\mathbf{G}}_j\bigl(\varsigma(z_j,\eta)\bigr)
=
\Delta\widehat{\mathbf{F}}(z_j,\eta).
\]

Using the FJFs2--Gauss quadrature formula \eqref{3.3a} in the $z$-variable,
we obtain
\begin{equation}\label{E3-step1}
\begin{split}
\|\mathbf{E}_3\|_{\varpi^{0,\beta-1,\mu}}
&=
\frac{1}{\Gamma(\beta)}
\Bigg(
\sum_{j=0}^{M}
\omega_j^{0,\beta-1,\mu}
\Bigg|
(1-z_j)^{\beta\mu}
\int_0^1
\varpi^{0,\beta-1,\mu}(\eta)\,
\\
&\qquad\qquad\qquad\qquad\times
\mathcal{I}_{\eta,M}^{0,\beta-1,\mu}
\Delta\widehat{\mathbf{F}}(z_j,\eta)\,d\eta
\Bigg|^2
\Bigg)^{1/2}.
\end{split}
\end{equation}

An application of the Cauchy--Schwarz inequality with respect to the weighted 
$\varpi^{0,\beta-1,\mu}(\eta)$ yields
\begin{equation}\label{E3-step2}
\begin{split}
\|\mathbf{E}_3\|_{\varpi^{0,\beta-1,\mu}}
&\le
\frac{1}{\Gamma(\beta)}
\Bigg(
\sum_{j=0}^{M}
\omega_j^{0,\beta-1,\mu}\,
(1-z_j)^{2\beta\mu}\,
\frac{1}{\beta}
\\
&\qquad\qquad\times
\int_0^1
\varpi^{0,\beta-1,\mu}(\eta)\,
\Big|
\mathcal{I}_{\eta,M}^{0,\beta-1,\mu}
\Delta\widehat{\mathbf{F}}(z_j,\eta)
\Big|^2\,d\eta
\Bigg)^{1/2}
\\
&=
\frac{1}{\Gamma(\beta+1)}
\Bigg(
\sum_{j=0}^{M}
\omega_j^{0,\beta-1,\mu}\,
\beta(1-z_j)^{\beta\mu}
\\
&\qquad\qquad\times
\Bigg[
(1-z_j)^{\beta\mu}
\int_0^1
\varpi^{0,\beta-1,\mu}(\eta)\,
\Big|
\mathcal{I}_{\eta,M}^{0,\beta-1,\mu}
\Delta\widehat{\mathbf{F}}(z_j,\eta)
\Big|^2\,d\eta
\Bigg]
\Bigg)^{1/2}.
\end{split}
\end{equation}

Therefore,
\begin{equation}\label{E3-step3}
\begin{split}
\|\mathbf{E}_3\|_{\varpi^{0,\beta-1,\mu}}
&\le
\frac{1}{\Gamma(\beta+1)}
\Bigg(
\sum_{j=0}^{M}
\omega_j^{0,\beta-1,\mu}\,
\beta(1-z_j)^{\beta\mu}
\Bigg)^{1/2}
\\
&\qquad\times
\max_{0\le j\le M}
\Bigg[
(1-z_j)^{\beta\mu}
\int_0^1
\varpi^{0,\beta-1,\mu}(\eta)\,
\Big|
\mathcal{I}_{\eta,M}^{0,\beta-1,\mu}
\Delta\widehat{\mathbf{F}}(z_j,\eta)
\Big|^2\,d\eta
\Bigg]^{1/2}.
\end{split}
\end{equation}

Using \eqref{E2-weight-bound}, we obtain
\begin{equation}\label{E3-step4}
\begin{split}
\|\mathbf{E}_3\|_{\varpi^{0,\beta-1,\mu}}
&\le
\frac{1}{\Gamma(\beta+1)}
\max_{0\le j\le M}
\Bigg[
(1-z_j)^{\beta\mu}
\int_0^1
\varpi^{0,\beta-1,\mu}(\eta)\,
\Big|
\mathcal{I}_{\eta,M}^{0,\beta-1,\mu}
\Delta\widehat{\mathbf{F}}(z_j,\eta)
\Big|^2\,d\eta
\Bigg]^{1/2}.
\end{split}
\end{equation}

Since
$\mathcal{I}_{\eta,M}^{0,\beta-1,\mu}\Delta\widehat{\mathbf{F}}(z_j,\eta)$
 the discrete identity
\eqref{eq:A4-right} gives
\begin{equation}\label{E3-step5}
\begin{aligned}
&(1-z_j)^{\beta\mu}
\int_0^1
\varpi^{0,\beta-1,\mu}(\eta)\,
\Big|
\mathcal{I}_{\eta,M}^{0,\beta-1,\mu}
\Delta\widehat{\mathbf{F}}(z_j,\eta)
\Big|^2\,d\eta
\\
&\qquad=
\int_{z_j}^{1}
\mu
\bigl((1-z_j)^{\mu}-(1-\varsigma)^{\mu}\bigr)^{\beta-1}
(1-\varsigma)^{\mu-1}
\Big|
{}_{z_j}\widetilde{\mathcal{I}}_{\varsigma,M}^{0,\beta-1,\mu}
\Delta\widehat{\mathbf{G}}_j(\varsigma)
\Big|^2\,d\varsigma
\\
&\qquad=
(1-z_j)^{\beta\mu}
\sum_{q=0}^{M}
\omega_q^{0,\beta-1,\mu}\,
\Big|
\Delta\widehat{\mathbf{G}}_j
\bigl(\varsigma(z_j,\eta_q^{0,\beta-1,\mu})\bigr)
\Big|^2.
\end{aligned}
\end{equation}

By the Lipschitz continuity of $\widehat{\mathbf{F}}$ with respect to its last
$Q$ arguments, we have
\begin{equation}\label{E3-step6}
\begin{aligned}
\Big|
\Delta\widehat{\mathbf{G}}_j
\bigl(\varsigma(z_j,\eta_q^{0,\beta-1,\mu})\bigr)
\Big|
&\le
L\,
\Big|
\mathbf{U}\bigl(\varsigma(z_j,\eta_q^{0,\beta-1,\mu})\bigr)
-
\mathbf{U}_M\bigl(\varsigma(z_j,\eta_q^{0,\beta-1,\mu})\bigr)
\Big|.
\end{aligned}
\end{equation}
Substituting this into \eqref{E3-step5}, we obtain
\begin{equation}\label{E3-step7}
\begin{aligned}
&(1-z_j)^{\beta\mu}
\int_0^1
\varpi^{0,\beta-1,\mu}(\eta)\,
\Big|
\mathcal{I}_{\eta,M}^{0,\beta-1,\mu}
\Delta\widehat{\mathbf{F}}(z_j,\eta)
\Big|^2\,d\eta
\\
&\qquad\le
L^2
(1-z_j)^{\beta\mu}
\sum_{q=0}^{M}
\omega_q^{0,\beta-1,\mu}\,
\Big|
\mathbf{U}\bigl(\varsigma(z_j,\eta_q^{0,\beta-1,\mu})\bigr)
-
\mathbf{U}_M\bigl(\varsigma(z_j,\eta_q^{0,\beta-1,\mu})\bigr)
\Big|^2.
\end{aligned}
\end{equation}

Since
$\mathcal{I}_{\eta,M}^{0,\beta-1,\mu}\mathbf{U}(\varsigma(z_j,\eta))$
interpolates $\mathbf{U}(\varsigma(z_j,\eta))$ at the nodes
$\eta_q^{0,\beta-1,\mu}$, we have
\[
\mathcal{I}_{\eta,M}^{0,\beta-1,\mu}\mathbf{U}
\bigl(\varsigma(z_j,\eta_q^{0,\beta-1,\mu})\bigr)
=
\mathbf{U}\bigl(\varsigma(z_j,\eta_q^{0,\beta-1,\mu})\bigr),
\qquad
q=0,\ldots,M.
\]
Hence, by another application of \eqref{eq:A4-right},
\begin{equation}\label{E3-step8}
\begin{aligned}
&(1-z_j)^{\beta\mu}
\sum_{q=0}^{M}
\omega_q^{0,\beta-1,\mu}\,
\Big|
\mathbf{U}\bigl(\varsigma(z_j,\eta_q^{0,\beta-1,\mu})\bigr)
-
\mathbf{U}_M\bigl(\varsigma(z_j,\eta_q^{0,\beta-1,\mu})\bigr)
\Big|^2
\\
&\qquad=
(1-z_j)^{\beta\mu}
\int_0^1
\varpi^{0,\beta-1,\mu}(\eta)\,
\Big|
\mathcal{I}_{\eta,M}^{0,\beta-1,\mu}
\mathbf{U}\bigl(\varsigma(z_j,\eta)\bigr)
-
\mathbf{U}_M\bigl(\varsigma(z_j,\eta)\bigr)
\Big|^2\,d\eta.
\end{aligned}
\end{equation}

Combining \eqref{E3-step4}, \eqref{E3-step7}, and \eqref{E3-step8}, we arrive at
\begin{equation}\label{E3-step9}
\begin{split}
\|\mathbf{E}_3\|_{\varpi^{0,\beta-1,\mu}}
&\le
\frac{L}{\Gamma(\beta+1)}
\max_{0\le j\le M}
\Bigg[
(1-z_j)^{\beta\mu}
\int_0^1
\varpi^{0,\beta-1,\mu}(\eta)\,
\\
&\qquad\qquad\times
\Big|
\mathcal{I}_{\eta,M}^{0,\beta-1,\mu}
\mathbf{U}\bigl(\varsigma(z_j,\eta)\bigr)
-
\mathbf{U}_M\bigl(\varsigma(z_j,\eta)\bigr)
\Big|^2\,d\eta
\Bigg]^{1/2}.
\end{split}
\end{equation}

We now split the term inside the norm as
\[
\mathcal{I}_{\eta,M}^{0,\beta-1,\mu}
\mathbf{U}\bigl(\varsigma(z_j,\eta)\bigr)
-
\mathbf{U}_M\bigl(\varsigma(z_j,\eta)\bigr)
=
\Bigl(
\mathcal{I}_{\eta,M}^{0,\beta-1,\mu}
\mathbf{U}\bigl(\varsigma(z_j,\eta)\bigr)
-
\mathbf{U}\bigl(\varsigma(z_j,\eta)\bigr)
\Bigr)
+
\mathbf{E}_M\bigl(\varsigma(z_j,\eta)\bigr).
\]
Therefore, by the triangle inequality,
\begin{equation}\label{E3-step10}
\begin{split}
\|\mathbf{E}_3\|_{\varpi^{0,\beta-1,\mu}}
&\le
\frac{L}{\Gamma(\beta+1)}
\max_{0\le j\le M}
\Bigg[
(1-z_j)^{\beta\mu}
\int_0^1
\varpi^{0,\beta-1,\mu}(\eta)\,
\\
&\qquad\qquad\times
\Big|
\mathcal{I}_{\eta,M}^{0,\beta-1,\mu}
\mathbf{U}\bigl(\varsigma(z_j,\eta)\bigr)
-
\mathbf{U}\bigl(\varsigma(z_j,\eta)\bigr)
\Big|^2\,d\eta
\Bigg]^{1/2}
\\
&\qquad+
\frac{L}{\Gamma(\beta+1)}
\max_{0\le j\le M}
\Bigg[
(1-z_j)^{\beta\mu}
\int_0^1
\varpi^{0,\beta-1,\mu}(\eta)\,
\big|
\mathbf{E}_M\bigl(\varsigma(z_j,\eta)\bigr)
\big|^2\,d\eta
\Bigg]^{1/2}.
\end{split}
\end{equation}

For the first term on the right-hand side, we apply
\eqref{eq:frac-4.5-like-right} with
\[
v(\varsigma)=\mathbf{U}(\varsigma),
\qquad
z=z_j,
\]
and obtain
\begin{equation}\label{E3-step11}
\begin{aligned}
&\max_{0\le j\le M}
\Bigg[
(1-z_j)^{\beta\mu}
\int_0^1
\varpi^{0,\beta-1,\mu}(\eta)\,
\Big|
\mathcal{I}_{\eta,M}^{0,\beta-1,\mu}
\mathbf{U}\bigl(\varsigma(z_j,\eta)\bigr)
-
\mathbf{U}\bigl(\varsigma(z_j,\eta)\bigr)
\Big|^2\,d\eta
\Bigg]^{1/2}
\\
&\qquad\le
C\,M^{-l}
\left\|
\partial_{\varsigma}^{\,l}
\mathbf{U}\bigl(1-(1-\varsigma)^{1/\mu}\bigr)
\right\|_{\varpi^{l,\beta-1+l,1}}.
\end{aligned}
\end{equation}

For the second term, using the change of variables $\varsigma=\varsigma(z_j,\eta)$, we have
\begin{equation}\label{E3-step12}
\begin{aligned}
&(1-z_j)^{\beta\mu}
\int_0^1
\varpi^{0,\beta-1,\mu}(\eta)\,
\big|
\mathbf{E}_M\bigl(\varsigma(z_j,\eta)\bigr)
\big|^2\,d\eta
\\
&\qquad=
\int_{z_j}^{1}
\mu
\bigl((1-z_j)^{\mu}-(1-\varsigma)^{\mu}\bigr)^{\beta-1}
(1-\varsigma)^{\mu-1}
\big|
\mathbf{E}_M(\varsigma)
\big|^2\,d\varsigma
\\
&\qquad\le
\|\mathbf{E}_M\|_{\varpi^{0,\beta-1,\mu}}^{2}.
\end{aligned}
\end{equation}
Consequently,
\begin{equation}\label{E3-step13}
\max_{0\le j\le M}
\Bigg[
(1-z_j)^{\beta\mu}
\int_0^1
\varpi^{0,\beta-1,\mu}(\eta)\,
\big|
\mathbf{E}_M\bigl(\varsigma(z_j,\eta)\bigr)
\big|^2\,d\eta
\Bigg]^{1/2}
\le
\|\mathbf{E}_M\|_{\varpi^{0,\beta-1,\mu}}.
\end{equation}

Substituting \eqref{E3-step11} and \eqref{E3-step13} into \eqref{E3-step10},
we obtain
\begin{equation}\label{E3-final-detailed}
\begin{split}
\|\mathbf{E}_3\|_{\varpi^{0,\beta-1,\mu}}
&\le
C\,M^{-l}
\left\|
\partial_{\varsigma}^{\,l}
\mathbf{U}\bigl(1-(1-\varsigma)^{1/\mu}\bigr)
\right\|_{\varpi^{l,\beta-1+l,1}}
\\
&\qquad+
\frac{L}{\Gamma(\beta+1)}
\|\mathbf{E}_M\|_{\varpi^{0,\beta-1,\mu}}.
\end{split}
\end{equation}

Finally, combining \eqref{eq:E1-correct}, \textcolor{black}{\eqref{E2-final-detailed}, and
\eqref{E3-final-detailed}, we arrive at
\begin{equation*}
\begin{split}
\|\mathbf{E}_M\|_{\varpi^{0,\beta-1,\mu}}
&\le
C\,M^{-l}
\Bigg(
\left\|
\partial_{\varsigma}^{\,l}
\mathbf{U}\bigl(1-(1-\varsigma)^{1/\mu}\bigr)
\right\|_{\varpi^{l,\beta-1+l,1}}
\\
&\qquad\qquad+
\max_{0\le j\le M}
\left\|
\partial_{\varsigma}^{\,l}
\widehat{\mathbf{F}}
\Bigl(
z_j,\,
1-(1-\varsigma)^{1/\mu},\,
\mathbf{U}\bigl(1-(1-\varsigma)^{1/\mu}\bigr)
\Bigr)
\right\|_{\varpi^{l,\beta-1+l,1}}
\Bigg)
\\
&\qquad+
\frac{L}{\Gamma(\beta+1)}
\|\mathbf{E}_M\|_{\varpi^{0,\beta-1,\mu}}.
\end{split}
\end{equation*}}
Since $L<\Gamma(\beta+1)$, the last term can be absorbed into the left-hand side,
and the estimate \eqref{conv-rate-main} follows.
\end{proof}

\section{Existence and uniqueness of the discrete collocation scheme}\label{secx5}

In this section, we prove that the collocation scheme introduced earlier
admits a well-defined and unique discrete solution.

To this end, we rewrite the discrete collocation system as a fixed-point
problem. Let $\mathbf{U}_M^{(m)}(z)$ denote the $m$th iterate. We define the
iteration by
\begin{equation}\label{3f.7a}
\begin{split}
\mathbf{U}_M^{(m)}(z)
&=
\frac{1}{\Gamma(\beta)}
\mathcal{I}_{z,M}^{0,\beta-1,\mu}
\Bigg[
(1-z)^{\beta\mu}
\int_0^1
\varpi^{0,\beta-1,\mu}(\eta)\,
\\
&\qquad\qquad\qquad\times
\mathcal{I}_{\eta,M}^{0,\beta-1,\mu}
\widehat{\mathbf{F}}
\bigl(
z,\varsigma(z,\eta),\mathbf{U}_M^{(m-1)}(\varsigma(z,\eta))
\bigr)\,d\eta
\Bigg].
\end{split}
\end{equation}

We now establish the existence and uniqueness of the discrete solution.

\begin{theorem}\label{thm:exist-unique}
Assume that $\widehat{\mathbf{F}}$ is Lipschitz continuous with respect to its
last $Q$ arguments, uniformly in the first two arguments, with Lipschitz
constant $L$.  If
\[
L<\Gamma(\beta+1),
\]
then the discrete collocation system \eqref{3.7a} admits a unique solution
$\mathbf{U}_M$. Moreover, the sequence
$\{\mathbf{U}_M^{(m)}\}_{m\ge0}$ generated by \eqref{3f.7a}
converges to $\mathbf{U}_M$ in the norm
$\|\cdot\|_{\varpi^{0,\beta-1,\mu}}$.
\end{theorem}

\begin{proof}
Let
\[
\widetilde{\mathbf{U}}_M^{(m)}(z)
:=
\mathbf{U}_M^{(m)}(z)-\mathbf{U}_M^{(m-1)}(z)
\]
denote the increment between two successive iterates.
Subtracting the relation for $m-1$ from \eqref{3f.7a}, we obtain
\begin{equation}\label{3jf.7a}
\begin{split}
\widetilde{\mathbf{U}}_M^{(m)}(z)
&=
\frac{1}{\Gamma(\beta)}
\mathcal{I}_{z,M}^{0,\beta-1,\mu}
\Bigg[
(1-z)^{\beta\mu}
\int_0^1
\varpi^{0,\beta-1,\mu}(\eta)\,
\mathcal{I}_{\eta,M}^{0,\beta-1,\mu}
\\
&\qquad\qquad\qquad\times
\Big(
\widehat{\mathbf{F}}
\bigl(
z,\varsigma(z,\eta),\mathbf{U}_M^{(m-1)}(\varsigma(z,\eta))
\bigr)
-
\widehat{\mathbf{F}}
\bigl(
z,\varsigma(z,\eta),\mathbf{U}_M^{(m-2)}(\varsigma(z,\eta))
\bigr)
\Big)\,d\eta
\Bigg].
\end{split}
\end{equation}

For convenience, define
\[
\Delta\widetilde{\widehat{\mathbf{F}}}(z,\eta)
:=
\widehat{\mathbf{F}}
\bigl(
z,\varsigma(z,\eta),\mathbf{U}_M^{(m-1)}(\varsigma(z,\eta))
\bigr)
-
\widehat{\mathbf{F}}
\bigl(
z,\varsigma(z,\eta),\mathbf{U}_M^{(m-2)}(\varsigma(z,\eta))
\bigr).
\]
Then \eqref{3jf.7a} becomes
\begin{equation}\label{ExistXi3_1}
\begin{split}
\|\widetilde{\mathbf{U}}_M^{(m)}\|_{\varpi^{0,\beta-1,\mu}}
&=
\frac{1}{\Gamma(\beta)}
\Big\|
\mathcal{I}_{z,M}^{0,\beta-1,\mu}
\Big[
(1-z)^{\beta\mu}
\int_0^1
\varpi^{0,\beta-1,\mu}(\eta)\,
\\
&\qquad\qquad\qquad\times
\mathcal{I}_{\eta,M}^{0,\beta-1,\mu}
\Delta\widetilde{\widehat{\mathbf{F}}}(z,\eta)\,d\eta
\Big]
\Big\|_{\varpi^{0,\beta-1,\mu}}.
\end{split}
\end{equation}

Applying the FJFs2--Gauss quadrature rule \eqref{3.3a} in the $z$-variable,
we obtain
\begin{equation}\label{ExistXi3_2}
\begin{split}
\|\widetilde{\mathbf{U}}_M^{(m)}\|_{\varpi^{0,\beta-1,\mu}}
&=
\frac{1}{\Gamma(\beta)}
\Bigg[
\sum_{j=0}^M
\omega_j^{0,\beta-1,\mu}
\Big|
(1-z_j)^{\beta\mu}
\int_0^1
\varpi^{0,\beta-1,\mu}(\eta)\,
\\
&\qquad\qquad\qquad\qquad\times
\mathcal{I}_{\eta,M}^{0,\beta-1,\mu}
\Delta\widetilde{\widehat{\mathbf{F}}}(z_j,\eta)\,d\eta
\Big|^2
\Bigg]^{1/2},
\end{split}
\end{equation}
where $z_j=z_j^{0,\beta-1,\mu}$.

Next, we apply the Cauchy--Schwarz inequality in the $\eta$-variable with
respect to the weight $\varpi^{0,\beta-1,\mu}(\eta)$. Since
\[
\int_0^1 \varpi^{0,\beta-1,\mu}(\eta)\,d\eta=\frac{1}{\beta},
\]
it follows from \eqref{ExistXi3_2} that
\begin{equation}\label{ExistXi3_3}
\begin{split}
\|\widetilde{\mathbf{U}}_M^{(m)}\|_{\varpi^{0,\beta-1,\mu}}
&\le
\frac{1}{\Gamma(\beta)}
\Bigg[
\sum_{j=0}^M
\omega_j^{0,\beta-1,\mu}
\frac{(1-z_j)^{2\beta\mu}}{\beta}
\int_0^1
\varpi^{0,\beta-1,\mu}(\eta)\,
\\
&\qquad\qquad\qquad\times
\Big|
\mathcal{I}_{\eta,M}^{0,\beta-1,\mu}
\Delta\widetilde{\widehat{\mathbf{F}}}(z_j,\eta)
\Big|^2\,d\eta
\Bigg]^{1/2}.
\end{split}
\end{equation}

Using the quadrature rule again in the $\eta$-variable, we arrive at
\begin{equation}\label{ExistXi3_4}
\begin{split}
\|\widetilde{\mathbf{U}}_M^{(m)}\|_{\varpi^{0,\beta-1,\mu}}
&\le
\frac{1}{\Gamma(\beta+1)}
\Bigg[
\sum_{j=0}^M
\omega_j^{0,\beta-1,\mu}\,
\beta(1-z_j)^{2\beta\mu}
\sum_{q=0}^M
\omega_q^{0,\beta-1,\mu}
\\
&\qquad\qquad\qquad\times
\big|
\Delta\widetilde{\widehat{\mathbf{F}}}(z_j,\eta_q)
\big|^2
\Bigg]^{1/2},
\end{split}
\end{equation}
where $\eta_q=\eta_q^{0,\beta-1,\mu}$.

By the Lipschitz continuity of $\widehat{\mathbf{F}}$, we have
\begin{equation}\label{ExistLipsch}
\big|
\Delta\widetilde{\widehat{\mathbf{F}}}(z_j,\eta_q)
\big|
\le
L\,
\big|
\widetilde{\mathbf{U}}_M^{(m-1)}
\bigl(\varsigma(z_j,\eta_q)\bigr)
\big|.
\end{equation}
Substituting \eqref{ExistLipsch} into \eqref{ExistXi3_4} yields
\begin{equation}\label{Xi3-L}
\begin{split}
\|\widetilde{\mathbf{U}}_M^{(m)}\|_{\varpi^{0,\beta-1,\mu}}
&\le
\frac{L}{\Gamma(\beta+1)}
\Bigg[
\sum_{j=0}^M
\omega_j^{0,\beta-1,\mu}\,
\beta(1-z_j)^{2\beta\mu}
\sum_{q=0}^M
\omega_q^{0,\beta-1,\mu}
\\
&\qquad\qquad\qquad\times
\big|
\widetilde{\mathbf{U}}_M^{(m-1)}
\bigl(\varsigma(z_j,\eta_q)\bigr)
\big|^2
\Bigg]^{1/2}.
\end{split}
\end{equation}

We now separate the $j$- and $q$-contributions. This gives
\begin{equation}\label{Xi3_5}
\begin{split}
\|\widetilde{\mathbf{U}}_M^{(m)}\|_{\varpi^{0,\beta-1,\mu}}
&\le
\frac{L}{\Gamma(\beta+1)}
\Bigg(
\sum_{j=0}^M
\omega_j^{0,\beta-1,\mu}\,
\beta(1-z_j)^{\beta\mu}
\Bigg)^{1/2}
\\
&\qquad\times
\max_{0\le j\le M}
\Bigg[
(1-z_j)^{\beta\mu}
\sum_{q=0}^M
\omega_q^{0,\beta-1,\mu}\,
\big|
\widetilde{\mathbf{U}}_M^{(m-1)}
\bigl(\varsigma(z_j,\eta_q)\bigr)
\big|^2
\Bigg]^{1/2}.
\end{split}
\end{equation}

Since \(0\le (1-z_j)^{\beta\mu}\le 1\), the quadrature weights are
positive, and the quadrature is exact for constants, we have
\[
\sum_{j=0}^M
\omega_j^{0,\beta-1,\mu}\,
\beta(1-z_j)^{\beta\mu}
\le
\beta\sum_{j=0}^M \omega_j^{0,\beta-1,\mu}
=
1.
\]
Hence,
\begin{equation}\label{Xi3_6}
\begin{split}
\|\widetilde{\mathbf{U}}_M^{(m)}\|_{\varpi^{0,\beta-1,\mu}}
&\le
\frac{L}{\Gamma(\beta+1)}
\max_{0\le j\le M}
\Bigg[
(1-z_j)^{\beta\mu}
\sum_{q=0}^M
\omega_q^{0,\beta-1,\mu}\,
\big|
\widetilde{\mathbf{U}}_M^{(m-1)}
\bigl(\varsigma(z_j,\eta_q)\bigr)
\big|^2
\Bigg]^{1/2}.
\end{split}
\end{equation}

To connect the discrete sum with the weighted norm, we use the interpolation
polynomial in the $\eta$-variable together with the discrete identity
\eqref{eq:A4-right}. Since
\[
\mathcal{I}_{\eta,M}^{0,\beta-1,\mu}
\widetilde{\mathbf{U}}_M^{(m-1)}
\bigl(\varsigma(z_j,\eta_q)\bigr)
=
\widetilde{\mathbf{U}}_M^{(m-1)}
\bigl(\varsigma(z_j,\eta_q)\bigr),
\qquad q=0,\ldots,M,
\]
we obtain
\begin{equation}\label{Xi3_7}
\begin{aligned}
&(1-z_j)^{\beta\mu}
\sum_{q=0}^M
\omega_q^{0,\beta-1,\mu}\,
\big|
\widetilde{\mathbf{U}}_M^{(m-1)}
\bigl(\varsigma(z_j,\eta_q)\bigr)
\big|^2
\\
&\qquad=
(1-z_j)^{\beta\mu}
\int_0^1
\varpi^{0,\beta-1,\mu}(\eta)\,
\Big|
\widetilde{\mathbf{U}}_M^{(m-1)}
\bigl(\varsigma(z_j,\eta)\bigr)
\Big|^2\,d\eta.
\end{aligned}
\end{equation}

Finally, applying the change of variables \(\varsigma=\varsigma(z_j,\eta)\), we get
\begin{equation}\label{Xi3_9}
\begin{aligned}
&(1-z_j)^{\beta\mu}
\int_0^1
\varpi^{0,\beta-1,\mu}(\eta)\,
\Big|
\widetilde{\mathbf{U}}_M^{(m-1)}
\bigl(\varsigma(z_j,\eta)\bigr)
\Big|^2\,d\eta
\\
&\qquad=
\int_{z_j}^{1}
\mu
\bigl((1-z_j)^{\mu}-(1-\varsigma)^{\mu}\bigr)^{\beta-1}
(1-\varsigma)^{\mu-1}
\big|
\widetilde{\mathbf{U}}_M^{(m-1)}(\varsigma)
\big|^2\,d\varsigma
\\
&\qquad\le
\|\widetilde{\mathbf{U}}_M^{(m-1)}\|_{\varpi^{0,\beta-1,\mu}}^2.
\end{aligned}
\end{equation}

Combining \eqref{Xi3_6}--\eqref{Xi3_9}, we obtain
\begin{equation}\label{Xi3-final}
\|\widetilde{\mathbf{U}}_M^{(m)}\|_{\varpi^{0,\beta-1,\mu}}
\le
\frac{L}{\Gamma(\beta+1)}
\|\widetilde{\mathbf{U}}_M^{(m-1)}\|_{\varpi^{0,\beta-1,\mu}}.
\end{equation}
Therefore, for all \(m\ge2\),
\[
\|\widetilde{\mathbf{U}}_M^{(m)}\|_{\varpi^{0,\beta-1,\mu}}
\le
\gamma\,
\|\widetilde{\mathbf{U}}_M^{(m-1)}\|_{\varpi^{0,\beta-1,\mu}},
\qquad
\gamma:=\frac{L}{\Gamma(\beta+1)}<1.
\]

It follows that \(\{\mathbf{U}_M^{(m)}\}_{m\ge0}\) is a Cauchy sequence in the
weighted norm \(\|\cdot\|_{\varpi^{0,\beta-1,\mu}}\), and hence it
converges to a limit \(\mathbf{U}_M\). Passing to the limit in
\eqref{3f.7a}, we conclude that \(\mathbf{U}_M\) is a solution of the
discrete collocation system \eqref{3.7a}.

The uniqueness follows from the contraction property \eqref{Xi3-final}, or
equivalently from the Banach fixed-point theorem. Therefore, the discrete
collocation scheme admits a unique solution.
\end{proof}

\section{Numerical results}\label{sec6}
In this section, we present three numerical examples to assess the accuracy, robustness, and computational efficiency of the proposed  spectral method for right-sided fractional differential equations with weak terminal singularities. The test problems are selected to illustrate the performance of the method for both linear and nonlinear systems, including cases with rational and irrational fractional orders.

The numerical examples are arranged progressively. We begin with a linear benchmark problem for which a direct comparison with an existing method from the literature is available. This allows us to quantify the improvement in accuracy achieved by the present formulation. We then consider nonlinear coupled systems in order to investigate the  convergence  of the proposed scheme in more challenging settings, particularly in the presence of irrational fractional orders and nonlinear interactions.

For all examples, the numerical errors are measured in the weighted norm
$\corr{L^{2}_{\omega^{0,\beta-1,\mu}}},$ with $\mu= \dfrac{\beta}{\,\left\lceil \beta / 0.18 \right\rceil\,}$
which is naturally consistent with the endpoint-adaptive approximation space introduced in the theoretical development. 

\textcolor{black}{\begin{example}\label{ex1}
We begin with the following linear benchmark problem:
\begin{equation}
{}_{t}^{C} D_{1}^{\beta} U(t) = U(t) + g(t), 
\quad t \in (0,1], 
\quad U(1) = 0,
\end{equation}
where
\begin{equation}
g(t) = \Gamma(1+\beta) - (1-t)^{\beta}, 
\quad 0 < \beta < 1.
\end{equation}
This problem admits the exact solution
\[
U(t) = (1-t)^{\beta}.
\]
\end{example}
\begin{table}
\centering
\caption{\textcolor{black}{Comparison of the weighted $\corr{L^{2}_{\omega^{0,\beta-1,\mu}}}$ error norms corresponding to the fractional orders $\beta=0.7,\,0.5,\,0.3$ in Example~\ref{ex1}.}}
\label{t1}
\begin{tabular}{@{}llll@{}}
\toprule
\multicolumn{4}{c}{The scheme \cite{tian2022high}} \\
\cmidrule{1-4}
$M$ & $\beta=0.7$ & $\beta=0.5$ & $\beta=0.3$ \\
\cmidrule{1-4}
64   & $4.5892\times 10^{-2}$ & $7.6437\times 10^{-2}$ & $9.1231\times 10^{-2}$ \\
128  & $2.8976\times 10^{-2}$ & $5.7148\times 10^{-2}$ & $8.4961\times 10^{-2}$ \\
256  & $1.8112\times 10^{-2}$ & $4.1959\times 10^{-2}$ & $7.6174\times 10^{-2}$ \\
512  & $1.1253\times 10^{-2}$ & $3.0444\times 10^{-2}$ & $6.6599\times 10^{-2}$ \\
1024 & $6.9671\times 10^{-3}$ & $2.1915\times 10^{-2}$ & $5.7213\times 10^{-2}$ \\
\cmidrule{1-4}
\multicolumn{4}{c}{Proposed scheme} \\
\cmidrule{1-4}
$M$ & $\beta=0.7$ & $\beta=0.5$ & $\beta=0.3$ \\
\cmidrule{1-4}
4  & $4.61605\times10^{-7}$  & $3.17694\times10^{-6}$  & $6.04998\times10^{-7}$ \\
8  & $1.55596\times10^{-10}$ & $6.38633\times10^{-11}$ & $9.43488\times10^{-10}$ \\
12 & $2.58913\times10^{-12}$ & $1.45948\times10^{-15}$ & $6.74423\times10^{-13}$ \\
16 & $1.23770\times10^{-13}$ & $5.06719\times10^{-16}$ & $2.04339\times10^{-14}$ \\
20 & $1.10689\times10^{-14}$ & $1.60550\times10^{-16}$ & $6.08317\times10^{-16}$ \\
\cmidrule{1-4}
\end{tabular}
\end{table}
Table~\ref{t1} provides a direct comparison between the proposed  spectral scheme and the method reported in \cite{tian2022high} for the linear benchmark problem. The results clearly demonstrate the superior performance of the present approach  using  a very small number of modes. This substantial reduction in the approximation error is a direct consequence of the singularity-preserving basis construction. Since the exact solution exhibits the weak terminal singular behavior, standard polynomial-based approximations are not expected to retain their spectral accuracy. By contrast, the present method explicitly incorporates the endpoint singularity into the approximation space, leading to a dramatic improvement in convergence. Moreover, the method performs uniformly well for all tested values of $\beta$. These results strongly confirm the effectiveness of the proposed formulation for linear right-sided fractional differential equations with non-smooth solutions.
}

%%%ex 2

 \textcolor{black}{\begin{example}\label{ex3} We consider the following linear system of right-sided fractional differential equations \cite{jankowski2015systems}: 
 \begin{equation}\label{eq7.2} \begin{cases} {}_{t}^{C}D_{1}^{\beta}\, U_1(t) = U_1(t)+U_2(t)+\sigma_1(t), \qquad t\in[0,1),\;\beta\in(0,1),\\[6pt] {}_{t}^{C}D_{1}^{\beta}\, U_2(t) = U_1(t)+U_2(t)+\sigma_2(t), \qquad t\in[0,1),\;\beta\in(0,1),\\[6pt] U_1(1)=0,\qquad U_2(1)=0, \end{cases} \end{equation} where  \[ \sigma_1(t) = \frac{\Gamma(3\beta+1)} {\Gamma(2\beta+1)} (1-t)^{2\beta} - (1-t)^{\beta} - (1-t)^{3\beta}, \] \[ \sigma_2(t) = \Gamma(\beta+1) - (1-t)^{\beta} - (1-t)^{3\beta}, \] and the terminal conditions are selected so that the system admits the exact solution \[ U_1(t)=(1-t)^{3\beta}, \qquad U_2(t)=(1-t)^{\beta}. \]   \end{example}     
 The numerical results in Tables~\ref{t4} and~\ref{t5}  confirm the excellent approximation properties of the proposed  spectral method. For all fractional order values of $\beta$, the computed weighted errors rapidly decrease  as $M$ increases.  A particularly notable observation is the strong performance of the method for the more singular cases, where the exact solutions possess pronounced endpoint singularities near $t=1$. The present basis successfully captures this behavior, thereby preserving the spectral convergence rate. 
\begin{table}
\centering
\caption{Numerical values of the $\corr{L^{2}_{\omega^{0,\beta-1,\mu}}}$-norm errors associated with $U_1(t)$ for $\beta=\sqrt{3}/2,\ \sqrt{3}/3,\ \sqrt{3}/4$ for Example~\ref{ex3}.}
\label{t4}
\begin{tabular}{@{}lccc@{}}
\toprule
$M$ & $\beta=\sqrt{3}/2$ & $\beta=\sqrt{3}/3$ & $\beta=\sqrt{3}/4$ \\
\midrule
10 & $3.67775\times 10^{-6}$ & $4.56276\times 10^{-7}$ & $8.72594\times 10^{-11}$ \\
15 & $7.64710\times 10^{-14}$ & $5.01912\times 10^{-15}$ & $4.09877\times 10^{-14}$ \\
20 & $3.44958\times 10^{-15}$ & $1.16949\times 10^{-16}$ & $9.36062\times 10^{-16}$ \\
25 & $3.01758\times 10^{-16}$ & $6.12509\times 10^{-18}$ & $4.85666\times 10^{-17}$ \\
\bottomrule
\end{tabular}
\end{table}
\begin{table}
\centering
\caption{Numerical values of the $\corr{L^{2}_{\omega^{0,\beta-1,\mu}}}$-norm errors associated with $U_2(t)$ for $\beta=\sqrt{3}/2,\ \sqrt{3}/3,\ \sqrt{3}/4$ for Example~\ref{ex3}.}
\label{t5}
\begin{tabular}{@{}lccc@{}}
\toprule
$M$ & $\beta=\sqrt{3}/2$ & $\beta=\sqrt{3}/3$ & $\beta=\sqrt{3}/4$ \\
\midrule
10 & $3.67775\times 10^{-6}$ & $4.56276\times 10^{-7}$ & $8.63266\times 10^{-11}$ \\
15 & $1.40945\times 10^{-13}$ & $5.74016\times 10^{-15}$ & $4.17436\times 10^{-14}$ \\
20 & $6.35797\times 10^{-15}$ & $1.33621\times 10^{-16}$ & $9.53423\times 10^{-16}$ \\
25 & $5.56174\times 10^{-16}$ & $6.99826\times 10^{-18}$ & $4.94673\times 10^{-17}$ \\
\bottomrule
\end{tabular}
\end{table}}

% ex 3

We next consider a nonlinear coupled system in order to assess the robustness of the proposed method for nonlinear right-sided fractional systems with irrational fractional orders.

\textcolor{black}{\begin{example}\label{ex2}
Consider the nonlinear fractional system \cite{jankowski2015systems}:
\begin{equation}\label{ex2.1}
\begin{cases}
{}_{t}^{C}D_{1}^{\beta}\, U_1(t)
=
\sin\!\big(U_1(t)\big)
+
\cos\!\big(U_2(t)\big)
+
f_1(t),
\quad U_1(1)=0,\\[6pt]
{}_{t}^{C}D_{1}^{\beta}\, U_2(t)
=
\sin^{2}\!\left(\dfrac{1}{2}U_1(t)\right)
+
2U_2(t)
+
f_2(t),
\quad U_2(1)=0,
\end{cases}
\end{equation}
where $t\in(0,1)$, $\beta\in(0,1)$, and the source terms $f_1(t)$ and $f_2(t)$ are defined by
\begin{equation}\label{ex2.1-f1}
\begin{aligned}
f_1(t)=\;&
\frac{\Gamma(3\beta+1)}{\Gamma(2\beta+1)}(1-t)^{2\beta}
+
\frac{6}{\Gamma(4-\beta)}(1-t)^{3-\beta} \\
&-
\sin\!\Big((1-t)^{3\beta}+(1-t)^3\Big)
-
\cos\!\Big((1-t)^{\beta}+(1-t)^{2+\beta}\Big),
\end{aligned}
\end{equation}
and
\begin{equation}\label{ex2.1-f2}
\begin{aligned}
f_2(t)=\;&
\Gamma(\beta+1)
+
\frac{\Gamma(3+\beta)}{2}(1-t)^2 \\
&-
\sin^{2}\!\left(\frac{1}{2}\Big((1-t)^{3\beta}+(1-t)^3\Big)\right)
-
2\Big((1-t)^{\beta}+(1-t)^{2+\beta}\Big).
\end{aligned}
\end{equation}
\end{example}
The exact solution is given as
\begin{equation}\label{ex2.1-exact}
U_1(t)=(1-t)^{3\beta}+(1-t)^3,
\qquad
U_2(t)=(1-t)^{\beta}+(1-t)^{2+\beta}.
\end{equation}
\begin{table}
\centering
\caption{Numerical values of the $\corr{L^{2}_{\omega^{0,\beta-1,\mu}}}$-norm errors associated with $U_1(t)$ for $\beta=\sqrt{3}/2,\ \sqrt{3}/3,\ \sqrt{3}/4$ for Example~\ref{ex2}.}
\label{t2}
\begin{tabular}{@{}lccc@{}}
\toprule
$M$ & $\beta=\sqrt{3}/2$ & $\beta=\sqrt{3}/3$ & $\beta=\sqrt{3}/4$ \\
\midrule
10 & $1.03863\times 10^{-5}$ & $7.28310\times 10^{-5}$ & $9.64950\times 10^{-5}$ \\
15 & $2.20443\times 10^{-10}$ & $3.31458\times 10^{-8}$ & $4.75540\times 10^{-8}$ \\
20 & $3.03739\times 10^{-15}$ & $1.67503\times 10^{-14}$ & $2.42205\times 10^{-14}$ \\
25 & $2.65700\times 10^{-16}$ & $5.30231\times 10^{-18}$ & $2.49176\times 10^{-17}$ \\
\bottomrule
\end{tabular}
\end{table}
\begin{table}
\centering
\caption{Numerical values of the $\corr{L^{2}_{\omega^{0,\beta-1,\mu}}}$-norm errors associated with $U_2(t)$ for $\beta=\sqrt{3}/2,\ \sqrt{3}/3,\ \sqrt{3}/4$ for Example~\ref{ex2}.}
\label{t3}
\begin{tabular}{@{}lccc@{}}
\toprule
$M$ & $\beta=\sqrt{3}/2$ & $\beta=\sqrt{3}/3$ & $\beta=\sqrt{3}/4$ \\
\midrule
10 & $2.33471\times 10^{-5}$ & $1.44425\times 10^{-4}$ & $1.91758\times 10^{-4}$ \\
15 & $2.14351\times 10^{-10}$ & $1.75928\times 10^{-8}$ & $2.72031\times 10^{-8}$ \\
20 & $9.35954\times 10^{-15}$ & $6.41618\times 10^{-15}$ & $1.17455\times 10^{-14}$ \\
25 & $8.18739\times 10^{-16}$ & $9.89657\times 10^{-18}$ & $4.23701\times 10^{-17}$ \\
\bottomrule
\end{tabular}
\end{table}
The numerical results reported in Tables~\ref{t2} and~\ref{t3} show that the weighted errors associated with  $U_1$ and $U_2$ decrease rapidly as the number of basis functions increases.  This behavior strongly indicates spectral-type convergence. More importantly, the method remains  highly accurate despite the nonlinear coupling and the presence of irrational fractional orders. This highlights one of the principal advantages of the present framework, namely, its ability to handle both rational and irrational orders.
}

\section{Conclusion}

In this work, we proposed a novel fractional Jacobi pseudospectral method for the numerical solution of low-regularity right-sided fractional systems of differential equations. The method is based on a specially designed family of fractional Jacobi basis functions that accurately preserves the terminal singular structure inherently generated by right-sided fractional derivatives. Owing to this singularity-preserving construction, the proposed scheme retains high-order accuracy even in the presence of weakly regular or non-smooth exact solutions, where standard spectral and finite-difference methods often exhibit a substantial loss of convergence. A rigorous theoretical analysis was established for the associated interpolation and collocation framework, from which explicit error estimates were derived. Moreover, the well-posedness of the resulting discrete system was proved, guaranteeing the existence and uniqueness of the numerical approximation. The proposed formulation successfully captures the intrinsic low-regularity behavior of the solution without requiring any smoothing transformations, singularity corrections, or auxiliary regularization procedures. The numerical experiments strongly support the theoretical results and demonstrate the efficiency, robustness, and high accuracy of the proposed method across a range of test problems with different regularity profiles. In particular, the method achieves highly accurate approximations with relatively few collocation points, confirming its computational effectiveness for right-sided fractional models.

\textcolor{black}{The present framework opens several promising directions for future research,
including extensions to multidimensional systems, nonlinear and tempered fractional
operators, distributed-order models, and inverse problems arising in fractional
dynamical systems. A further direction concerns the possible design of fast
evaluation variants of the proposed scheme, in the spirit of the fast methods
developed for left-sided Caputo derivatives
\cite{jiang2017fast,yan2017fast,zhang2022fast,fang2020fast,huang2022unified}. As the
present pseudospectral formulation is a global, non-marching scheme, it does not
rely on the time-history recursion that motivates such fast methods in the
finite-difference setting; whether an analogous acceleration strategy could be
formulated for the present right-sided framework, and what efficiency gains it
might offer, is left for future investigation.}

\section*{Author contributions}
I.G.A., M.H., and M.R. designed the research.
M.A.Z. conducted the numerical experiments and prepared the first draft of the manuscript.
I.G.A., M.H., M.R., and M.A.Z. developed the analysis and reviewed the manuscript.

\section*{Data availability statement}
No datasets were generated or analyzed during the current study.

\section*{Declarations}

\section*{Conflict of interest}
The authors declare that they have no conflict of interest.

	\section*{Funding} This work was supported and funded by the Deanship of Scientific Research at Imam Mohammad Ibn Saud Islamic University (IMSIU) (grant number IMSIU-DDRSP2602).

\bibliographystyle{elsart-num-sort}
		%% \bibliography{mrabbrev,ref,ref2}
		
		%%%%%%%%%%%%%%%%%%%%%%%%%%%%%%%
		
		\bibliography{Bibfileamc}

@article{tian2022high,
  title={A high-order numerical scheme for right Caputo fractional differential equations with uniform accuracy.},
  author={Tian, Li and Wang, Ziqiang and Cao, Junying},
  journal={Electronic Research Archive},
  volume={30},
  number={10},
  year={2022}
}

@article{zhou2016well,
  title={Well-posedness of a kind of nonlinear coupled system of fractional differential equations},
  author={Zhou, XiaoJun and Xu, ChuanJu},
  journal={Science China Mathematics},
  volume={59},
  number={6},
  pages={1209--1220},
  year={2016},
  publisher={Springer}
}

@article{zayernouri1945spectral,
  title={Spectral and spectral element methods for fractional ordinary and partial differential equations},
  author={Zayernouri, Mohsen and Wang, Li-Lian and Shen, Jie and Karniadakis, George Em},
  journal={Cambridge Monographs on Applied and Computational Mathematics (Cambridge University Press, 2024)},
  year={1945}
}

@article{mokhtary2020computational,
  title={A computational approach for the non-smooth solution of non-linear weakly singular Volterra integral equation with proportional delay},
  author={Mokhtary, Payam and Moghaddam, Behrouz Parsa and Lopes, Ant{\'o}nio M and Machado, JA Tenreiro},
  journal={Numerical algorithms},
  volume={83},
  number={3},
  pages={987--1006},
  year={2020},
  publisher={Springer}
}

@article{sgura2014finite,
  title={A finite difference approach for the numerical solution of non-smooth problems for Boundary Value ODEs},
  author={Sgura, Ivonne},
  journal={Mathematics and Computers in Simulation},
  volume={95},
  pages={146--162},
  year={2014},
  publisher={Elsevier}
}

@book{podlubny1998fractional,
  title={Fractional differential equations: an introduction to fractional derivatives, fractional differential equations, to methods of their solution and some of their applications},
  author={Podlubny, Igor},
  volume={198},
  year={1998},
  address		= {Elsevier},
  publisher={Elsevier}
}

@article{zaky2020priori,
  title={A priori error estimates of a {J}acobi spectral method for nonlinear systems of fractional boundary value problems and related {V}olterra-{F}redholm integral equations with smooth solutions},
  author={Zaky, Mahmoud A and Ameen, Ibrahem G},
  journal={Numerical Algorithms},
  volume={84},
  number={1},
  pages={63--89},
  year={2020},
  publisher={Springer}
}

@article{zayernouri2013fractional,
  title={Fractional {S}turm--{L}iouville eigen-problems: theory and numerical approximation},
  author={Zayernouri, Mohsen and Karniadakis, George Em},
  journal={Journal of Computational Physics},
  volume={252},
  pages={495--517},
  year={2013},
  publisher={Elsevier}
}

@article{zaky2020accurate,
  title={An accurate spectral collocation method for nonlinear systems of fractional differential equations and related integral equations with nonsmooth solutions},
  author={Zaky, Mahmoud A},
  journal={Applied Numerical Mathematics},
  volume={154},
  pages={205--222},
  year={2020},
  publisher={Elsevier}
}

@article{zaky2020semi,
  title={Semi-implicit {G}alerkin--{L}egendre spectral schemes for nonlinear time-space fractional diffusion--reaction equations with smooth and nonsmooth solutions},
  author={Zaky, Mahmoud A and Hendy, Ahmed S and Mac{\'\i}as-D{\'\i}az, Jorge E},
  journal={Journal of Scientific Computing},
  volume={82},
  pages={1--27},
  year={2020},
  publisher={Springer}
}

@book{tarasov2011fractional,
  title={Fractional dynamics: applications of fractional calculus to dynamics of particles, fields and media},
  author={Tarasov, Vasily E},
  year={2011},
  publisher={Springer Science \& Business Media}
}

@article{caboussat2022least,
  title={A least-squares method for the solution of the non-smooth prescribed Jacobian equation},
  author={Caboussat, Alexandre and Glowinski, Roland and Gourzoulidis, Dimitrios},
  journal={Journal of Scientific Computing},
  volume={93},
  number={1},
  pages={15},
  year={2022},
  publisher={Springer}
}

@article{faghih2022non,
  title={Non-linear system of multi-order fractional differential equations: theoretical analysis and a robust fractional Galerkin implementation},
  author={Faghih, Amin and Mokhtary, P},
  journal={Journal of Scientific Computing},
  volume={91},
  number={2},
  pages={35},
  year={2022},
  publisher={Springer}
}

@article{zheng2020numerical,
  title={Numerical approximations for the variable coefficient fractional diffusion equations with non-smooth data},
  author={Zheng, Xiangcheng and Ervin, Vincent J and Wang, Hong},
  journal={Computational Methods in Applied Mathematics},
  volume={20},
  number={3},
  pages={573--589},
  year={2020},
  publisher={De Gruyter}
}

@article{faghih2025spectral,
  title={On a spectral solver for highly oscillatory and non-smooth solutions of a class of linear fractional differential systems},
  author={Faghih, Amin},
  journal={arXiv preprint arXiv:2511.06410},
  year={2025}
}

@article{zaky2025improved,
  title={Improved spectral method for the nonsmooth solutions of nonlinear fractional differential equations},
  author={Zaky, Mahmoud A and Abu Arqub, Omar},
  journal={Journal of Applied Mathematics and Computing},
  pages={1--18},
  year={2025},
  publisher={Springer}
}

@article{ameen2021pseudo,
  title={A pseudo-spectral scheme for systems of two-point boundary value problems with file and right sided fractional derivatives and related integral equations},
  author={Ameen, IG and Elkot, NA and Zaky, MA and Hendy, AS and Doha, EH},
  journal={Computer Modeling in Engineering \& Sciences},
  volume={128},
  number={1},
  pages={21--41},
  year={2021},
  publisher={Tech Science Press}
}

@article{zaky2024non,
  title={Non-polynomial collocation spectral scheme for systems of nonlinear Caputo--Hadamard differential equations},
  author={Zaky, Mahmoud A and Ameen, Ibrahem G and Babatin, Mohammed and Akg{\"u}l, Ali and Hammad, Magda and Lopes, Ant{\'o}nio M},
  journal={Fractal and Fractional},
  volume={8},
  number={5},
  pages={262},
  year={2024},
  publisher={MDPI}
}

@article{jankowski2015systems,
  title={Systems of nonlinear fractional differential equations},
  author={Jankowski, Tadeusz},
  journal={Fractional Calculus and Applied Analysis},
  volume={18},
  number={1},
  pages={122--132},
  year={2015},
  publisher={Springer}
}

@article{shiri2021terminal,
  title={Terminal value problems for the nonlinear systems of fractional differential equations},
  author={Shiri, Babak and Wu, Guo--Cheng and Baleanu, Dumitru},
  journal={Applied Numerical Mathematics},
  volume={170},
  pages={162--178},
  year={2021},
  publisher={Elsevier}
}

@article{zaky2025high,
  title={High-order fractional spectral collocation method for the nonsmooth solutions of nonlinear right-sided Caputo terminal value problems},
  author={Zaky, Mahmoud A and Ameen, Ibrahem G and Arqub, Omar Abu and Doha, Eid H},
  journal={Fractional Calculus and Applied Analysis},
  pages={1--22},
  year={2025},
  publisher={Springer}
}

@article{jiang2017fast,
  title={Fast evaluation of the Caputo fractional derivative and its applications to fractional diffusion equations},
  author={Jiang, Shidong and Zhang, Jiwei and Zhang, Qian and Zhang, Zhimin},
  journal={Communications in Computational Physics},
  volume={21},
  number={3},
  pages={650--678},
  year={2017},
  publisher={Cambridge University Press}
}

@article{yan2017fast,
  title={Fast evaluation of the Caputo fractional derivative and its applications to fractional diffusion equations: a second-order scheme},
  author={Yan, Yonggui and Sun, Zhi-Zhong and Zhang, Jiwei},
  journal={Communications in Computational Physics},
  volume={22},
  number={4},
  pages={1028--1048},
  year={2017},
  publisher={Cambridge University Press}
}

@article{zhang2022fast,
  title={Fast Second-Order Evaluation for Variable-Order Caputo Fractional Derivative with Applications to Fractional Sub-Diffusion Equations},
  author={Zhang, Jia-Li and Fang, Zhi-Wei and Sun, Hai-Wei},
  journal={Numerical Mathematics: Theory, Methods and Applications},
  volume={15},
  number={1},
  pages={200--226},
  year={2022}
}

@article{fang2020fast,
  title={A fast method for variable-order Caputo fractional derivative with applications to time-fractional diffusion equations},
  author={Fang, Zhi-Wei and Sun, Hai-Wei and Wang, Hong},
  journal={Computers \& Mathematics with Applications},
  volume={80},
  number={5},
  pages={1443--1458},
  year={2020},
  publisher={Elsevier}
}

@article{huang2022unified,
  title={A unified fast memory-saving time-stepping method for fractional operators and its applications},
  author={Huang, Yuxiang and Li, Qiaoge and Li, Rongxin and Zeng, Fanhai and Guo, Ling},
  journal={Numerical Mathematics: Theory, Methods and Applications},
  volume={15},
  number={3},
  pages={679--714},
  year={2022}
}

%\bibliography{Bibfileamc}% common bib file

\end{document}